\documentclass{imsart}
 \pdfoutput=1
\usepackage{amsmath,amssymb,bbm,amsthm,graphicx,subcaption,booktabs}
\usepackage[numbers]{natbib}
\usepackage{float}
\usepackage{soul}
\usepackage{romannum}
\usepackage{multirow}
\usepackage{csquotes}
\usepackage{gensymb}
\usepackage{mathtools}
\RequirePackage[colorlinks,citecolor=blue,urlcolor=blue]{hyperref}

\DeclareMathOperator*{\argmax}{arg\,max}
\DeclareMathOperator*{\argmin}{arg\,min}
\DeclareMathOperator{\diag}{diag}
\newcommand{\RN}[1]{%
  \textup{\uppercase\expandafter{\romannumeral#1}}%
}
\newcommand*{\rom}[1]{\expandafter\@slowromancap\romannumeral #1@}

\theoremstyle{plain}
\newtheorem{theorem}{Theorem}
\numberwithin{table}{section}
\newtheorem{lemma}[theorem]{Lemma}

\usepackage{enumerate}

\newtheorem{definition}{Definition}[section]
\newcommand{\minus}{\scalebox{0.75}[1.0]{$-$}}
\usepackage{accents}
\usepackage[utf8]{inputenc}
\usepackage[english]{babel}
\usepackage{mathrsfs}
\usepackage{eufrak}
\newlength{\dhatheight}
\newcommand{\doublehat}[1]{%
    \settoheight{\dhatheight}{\ensuremath{\hat{#1}}}%
    \addtolength{\dhatheight}{-0.35ex}%
    \hat{\vphantom{\rule{1pt}{\dhatheight}}%
    \smash{\hat{#1}}}}
    \DeclareMathOperator{\sign}{sign}
\newcommand*\diff{\mathop{}\!\mathrm{d}}
\usepackage{etoolbox}
\makeatletter
\g@addto@macro{\definition}{\upshape}
\makeatother

\usepackage[shortlabels]{enumitem}

\numberwithin{equation}{section}
\numberwithin{theorem}{section}

\begin{document}
\pagenumbering{arabic}
\begin{frontmatter}

\title{Bayesian Inference on Multivariate Medians and Quantiles}

\begin{aug}
\author{\fnms{Indrabati}  \snm{Bhattacharya}\ead[label=e1]{ibhatta@ncsu.edu}}
\and
\author{\fnms{Subhashis} \snm{Ghosal}\ead[label=e2]{sghosal@ncsu.edu}}
  \runtitle{Bayesian Inference on Multivariate Medians}
\affiliation{North Carolina State University}

  \address{Department of Statistics\\ North Carolina State University\\ 
          \printead{e1,e2}}
  \runauthor{I. Bhattacharya and S. Ghosal}

 \end{aug}

\begin{abstract}
In this paper, we consider Bayesian inference on a class of multivariate median and the multivariate quantile functionals of a joint distribution using a Dirichlet process prior. Since, unlike univariate quantiles, the exact posterior distribution of multivariate median and multivariate quantiles are not obtainable explicitly, we study these distributions asymptotically. We derive a Bernstein-von Mises theorem for the multivariate $\ell_1$-median with respect to general $\ell_p$-norm, which in particular shows that its posterior concentrates around its true value at $n^{-1/2}$-rate and its credible sets have asymptotically correct frequentist coverage. In particular, asymptotic normality results for the empirical multivariate median with general $\ell_p$-norm is also derived in the course of the proof which extends the results from the case $p=2$ in the literature to a general $p$. The technique involves approximating the posterior Dirichlet process by a Bayesian bootstrap process and deriving a conditional Donsker theorem. We also obtain analogous results for an affine equivariant version of the multivariate $\ell_1$-median based on an adaptive transformation and re-transformation technique. The results are extended to a joint distribution of multivariate quantiles. The accuracy of the asymptotic result is confirmed by a simulation study. We also use the results to obtain Bayesian credible regions for multivariate medians for Fisher's iris data, which consists of four features measured for each of three plant species.
\end{abstract}

\begin{keyword}
\kwd{Affine equivariance}
\kwd{Bayesian bootstrap}
\kwd{Donsker class}
\kwd{Dirichlet process}
\kwd{Empirical process}
\kwd{Geometric quantile}
\kwd{Multivariate median}
\end{keyword}

\end{frontmatter}
\section{Introduction}
Various definitions of a sample multivariate median for a set of multivariate observations  have been proposed and studied in the literature. One of the most popular versions of multivariate median is called the multivariate $\ell_1$-median. For a set of sample points $X_1,X_2,\dots,X_n \in \mathbb{R}^k$, $k\geq 2$, the sample $\ell_1$-median is obtained by minimizing $n^{-1}\sum_{i=1}^n \Vert X_i-\theta \Vert$ with respect to $\theta$, where $\Vert \cdot \Vert$ denotes some norm. The most commonly used norm is the $\ell_p$-norm $\Vert x \Vert_p=\left(\sum_{j=1}^k \vert x_j \vert ^p\right)^{1/p}$, $1 \leq p \leq \infty$. The most popular version of the $\ell_1$-median that uses the usual Euclidean norm $\Vert x \Vert_2 = \left(\sum_{j=1}^k x_j^2\right)^{1/2}$ is known as the spatial median. This corresponds to $p=2$, in which case we drop $p$ from the notation $\Vert \cdot \Vert_p$. Clearly the case $p=1$ gives rise to the vector of coordinatewise medians. The sample $\ell_1$-median with $\ell_p$-norm is given by
\begin{equation}
\hat{\theta}_{n;p}=\argmin_{\theta}\frac{1}{n}\sum_{i=1}^n\Vert X_i-\theta \Vert_p.
\end{equation}
The spatial median is a highly robust estimator of the location and its breakdown point is $1/2$ which is as high as that of the coordinatewise median (see Lopuhaa and Rousseeuw \citep{lopuhaa1991breakdown} for more details). Brown \citep{brown1983statistical} proved the asymptotic normality of the sample spatial median when $k=2$. Brown \citep{brown1983statistical} also showed that when the sample comes from a multivariate spherical normal distribution, the asymptotic efficiency of the spatial median relative to the sample mean vector tends to $1$ as $k \rightarrow \infty$. Gower \citep{gower1974algorithm} gave the codes for computing the spatial median. Weiszfeld \citep{weiszfeld1937point} devised an iterative method for computing the spatial median. Kuhn \citep{kuhn1973note} showed that Weiszfeld's \citep{weiszfeld1937point} algorithm does not converge if the starting point is inside the domain of attraction of the data points. Vardi and Zhang \citep{vardi2000multivariate} modified Weiszfeld's \citep{weiszfeld1937point} algorithm and showed that the modified algorithm converges for any starting point. \par 
The $\ell_1$-median functional of a probability distribution $P$ based on the $\ell_p$-norm is given by
\begin{equation}
\theta_p(P)= \argmin_{\theta}P\left(\Vert X - \theta \Vert_p - \Vert X \Vert_p\right),
\end{equation}
for $Pf = \int f \diff P$ and $1 \leq p \leq \infty$. It can be noted that this definition does not require any moment assumption on $X$, since $\vert\Vert X - \theta \Vert_p - \Vert X \Vert_p\vert \leq \Vert \theta \Vert_p$. Henceforth, we fix $1<p<\infty$ and drop $p$ from the notations $\hat{\theta}_{n;p}$ and $\theta_p(P)$ and just write $\hat{\theta}_n$ and $\theta(P)$ respectively.

We wish to infer about $\theta(P)$, which is a functional of the underlying and unknown random distribution $P$, given a set of observed data. The most commonly used prior on a random distribution $P$ is the Dirichlet process prior which we discuss in Section 2. In the univariate case, the exact posterior distribution of the median functional can be derived explicitly (see Chapter 4, Ghosal and van der Vaart \citep{ghosal2017fundamentals} for more details). Unfortunately, in the multivariate case, the exact posterior distribution can only be computed by simulations. The posterior distribution can be used to compute point estimates and credible sets. It is of interest to study the frequentist accuracy of the Bayesian estimator and frequentist coverage of posterior credible regions. In the parametric context, the Bernstein-von Mises theorem ensures that the Bayes estimator converges at the parametric rate $n^{-1/2}$ and a Bayesian $(1-\alpha)$ credible set has asymptotic frequentist coverage $(1-\alpha)$. Interestingly, a functional version of the Bernstein-von Mises theorem holds for the distribution under the Dirichlet process prior as shown by Lo (\citep{lo1983weak}, \citep{lo1986remark}). A functional Bernstein-von Mises theorem can potentially establish Bernstein-von Mises theorem for certain functionals. We study posterior concentration properties of the multivariate $\ell_1$-median $\theta(P)$ and show that the posterior distribution of $\theta(P)$ centered at the sample $\ell_1$-median $\hat{\theta}_n$ is asymptotically normal. We also note that this asymptotic distribution matches with the asymptotic distribution of $\hat{\theta}_n$ centered at the true value $\theta_0 \equiv\theta(P_0)$, where $P_0$ is the true value of $P$, thus proving a Bernstein-von Mises theorem for the multivariate $\ell_1$-median.\par
One possible shortcoming of the multivariate $\ell_1$-median is that it lacks equivariance under affine transformation of the data. Chakraborty, Chaudhuri and Oja \citep{chakraborty1998operating} proposed an affine-equivariant modification of the sample spatial median using a data-driven transformation and re-transformation technique. There is no population analog of this modified median. We define a Bayesian analog of this modified median in the following way. We put a Dirichlet process prior on the distribution of a transformed data depending on the observed data and induce the posterior distribution on $\theta(P)$ to make its distribution translation equivariant. We show that the asymptotic posterior distribution of $\theta(P)$ thus obtained centered at the affine-equivariant multivariate median estimate matches with the asymptotic distribution of the latter centered at $\theta_0$, while both the limiting distributions are normal.\par
Chaudhuri \citep{chaudhuri1996geometric} introduced the notion of geometric quantile as a generalization of the spatial median. For the univariate case it is easy to see that for $X_1,\dots,X_n \in \mathbb{R}$ and $u=2\alpha-1$, the sample $\alpha$-quantile $\hat{Q}_n(u)$ is obtained by minimizing $\sum_{i=1}^n\{\vert X_i-\xi \vert + u(X_i-\xi)\}$ with respect to $\xi$. Chaudhuri \citep{chaudhuri1996geometric} indexed the $k$-variate quantiles by points in the open unit ball $B^{(k)}\coloneqq\{u : u \in \mathbb{R}^k, \Vert u \Vert < 1\}$. For any $u \in B^{(k)}$, Chaudhuri \citep{chaudhuri1996geometric} defined the sample geometric $u$-quantile by minimizing $\sum_{i=1}^n \{\Vert X_i-\xi\Vert + \langle u, X_i-\xi \rangle\}$ with respect to $\xi$. Generalizing Chaudhuri's \citep{chaudhuri1996geometric} definition of multivariate quantile based on the $\ell_2$-norm to the $\ell_p$-norm with $1<p<\infty$, we define the multivariate sample quantile process as 
\begin{equation}
\hat{Q}_n(u)= \argmin_{\xi \in \mathbb{R}^k}\frac{1}{n}\sum_{i=1}^n\Phi_p(u, X_i-\xi),
\end{equation}
where $\Phi_p(u,t)=\Vert t \Vert_p + \langle u,t \rangle$ with $u \in B_q^{(k)}\coloneqq\{u : u \in \mathbb{R}^k, \Vert u \Vert_q < 1\}$ and $q$ is the conjugate index of $p$, i.e., $p^{-1}+q^{-1}=1$. It is easy to see that $\hat{Q}_n(0)$ coincides with the sample multivariate $\ell_1$-median $\hat{\theta}_n$. Similarly, for $u \in B_q^{(k)}$, the multivariate quantile process of a probability measure $P$ is given by
\begin{equation}
Q_P(u)= \argmin_{\xi \in \mathbb{R}^k}P\{\Phi_p(u,X-\xi)-\Phi_p(u,X)\}.
\end{equation}
with $Q_0(u)\equiv Q_{P_0}(u)$ being the true multivariate quantile function. We prove that, with $P$ having a Dirichlet process prior and for finitely many $u_1,\dots,u_m$, the joint distribution of $\{\sqrt{n}(Q_P(u_1)-\hat{Q}_n(u_1)),\dots,\sqrt{n}(Q_P(u_m)-\hat{Q}_n(u_m))\}$ given the data, converges to a multivariate normal distribution. Moreover, it is also noted that the joint distribution of $\{\sqrt{n}(\hat{Q}_n(u_1)-Q_0(u_1)),\dots,\sqrt{n}(\hat{Q}_n(u_m)-Q_0(u_m))\}$ converges to the same multivariate normal distribution.\par 
The rest of this paper is organized as follows. In Section 2, we give the background needed to introduce the main results. In Section 3, we state the Bernstein-von Mises theorem for the multivariate $\ell_1$-median and the theorems we need to prove the same. In Section 4, we give the Bernstein-von Mises theorem for the affine-equivariant $\ell_1$-median. In Section 5, we state the Bernstein-von Mises theorem related to the multivariate quantiles. In Section 6, we present a simulation study and an analysis of Fisher's iris data. All the proofs are given in Section 7.
\section{Background and Preliminaries}
Before giving the background, we introduce some notations that we follow in this paper. Throughout this paper, $\mathrm{N}_k(\mu,\Sigma)$ denotes a $k$-variate multivariate normal distribution with mean vector $\mu$ and covariance matrix $\Sigma$. Also, $\mathrm{DP}(\alpha)$ denotes a Dirichlet process with centering measure $\alpha$ (see Chapter 4, Ghosal and van der Vaart \citep{ghosal2017fundamentals} for more details).\par
Here $\rightsquigarrow$ and $\overset{P}\to$ denote weak convergence i.e. convergence in distribution and convergence in probability respectively. For a sequence $X_n$, the notation $X_n=O_P(a_n)$ means that $X_n/a_n$ is stochastically bounded.\par
Let $X_i \in \mathbb{R}^k$, $i=1,\dots, n$, be independently and identically distributed observations from a $k$-variate distribution $P$ and let $P$ have the $\mathrm{DP}(\alpha)$ prior. The parameter space is taken to be $\mathbb{R}^k$. The Bayesian model is given by
\begin{equation}
X_1,X_2,\dots, X_n\vert P \overset{\text{iid}}{\sim} P, \quad P \sim \mathrm{DP}(\alpha).
\end{equation}
The posterior distribution of $P$ given $X_1,X_2,\dots,X_n$ is $\mathrm{DP}(\alpha + n\mathbb{P}_n)$, where $\mathbb{P}_n=n^{-1}\sum_{i=1}^n \delta_{X_i}$ is the empirical measure (see Chapter 4, Ghosal and van der Vaart \citep{ghosal2017fundamentals} for more details). The posterior distribution $\mathrm{DP}(\alpha+n\mathbb{P}_n)$ can be expressed as $V_nQ + (1-V_n)\mathbb{B}_n$, where the variables $Q \sim \mathrm{DP}(\alpha)$, $\mathbb{B}_n \sim \mathrm{DP}(n\mathbb{P}_n)$ and $V_n \sim \mathrm{Be}(\vert \alpha \vert,n)$ are independent and $\mathrm{Be}(a,b)$ denotes a beta distribution with parameters $a$ and $b$. The process $\mathbb{B}_n$ is also known as the Bayesian bootstrap distribution and can be defined by the linear operator $\mathbb{B}_nf = \sum_{i=1}^nB_{ni}f(X_i)$, where $(B_{n1},B_{n2},\dots,B_{nn})$ is a random vector following the Dirichlet distribution $\mathrm{Dir}(n;1,1,\dots,1)$.\par
If $P \vert X_1,\dots,X_n \sim \mathrm{DP}(\alpha + n\mathbb{P}_n)$, then $\Vert P-\mathbb{B}_n \Vert_{\mathrm{TV}}=o_{\mathrm{Pr}}(n^{-1/2})$ a.s. $[P_0^{\infty}]$, where $\mathrm{Pr}\equiv P^{\infty}\times \mathbb{B}_n$ (see Chapter 12, Ghosal and van der Vaart \citep{ghosal2017fundamentals}). Using this fact, we show that conditional on the data, $\mathcal{L}(\sqrt{n}(\theta(P)-\hat{\theta}_n)|P \sim \mathrm{DP}(\alpha + n\mathbb{P}_n))$ imitates $\mathcal{L}(\sqrt{n}(\theta(\mathbb{B}_n)-\hat{\theta}_n))$ asymptotically, where $\theta(\mathbb{B}_n)=\argmin_{\theta} \mathbb{B}_n \Vert X - \theta \Vert_p$ and $\mathcal{L}(X)$ denotes the distribution of $X$. Thus we conclude that the asymptotic posterior distribution of $\sqrt{n}(\theta(P) - \hat{\theta}_n)$ is the same as the distribution of $\sqrt{n}(\theta(\mathbb{B}_n) - \hat{\theta}_n)$ conditional on the data. \par
The estimator $\hat{\theta}_n$ can be viewed as a $Z$-estimator, i.e. it satisfies the system of equations $\Psi_n(\theta)=\mathbb{P}_n\psi(\cdot,\theta)=0$, where $\psi(\cdot,\theta)=(\psi_1(\cdot,\theta),\dots,\psi_k(\cdot,\theta))^T$ is a $k \times 1$ vector of functions from $\mathbb{R}^k \times \mathbb{R}^k$ to $\mathbb{R}$ with
\begin{equation}
\psi_j(x,\theta) =\frac{\vert x_j-\theta_j \vert^{p-1}}{\Vert x-\theta \Vert_p^{p-1}}\sign(\theta_j-x_j), \quad j=1,\dots,k.
\end{equation}
Observe that $\theta_0 \equiv\theta(P_0)$ satisfies the system of equations $\Psi_0(\theta)=P_0\psi(\cdot,\theta)=0$. Also define
$\dot{\Psi}_0=[\partial{\Psi_0}/{\partial{\theta}}]_{\theta=\theta_0}$
and $\Sigma_0=P_0\psi(X,\theta_0)\psi^{T}(X,\theta_0)$. 
\section{Bernstein-von Mises theorem for $\ell_1$-median}
\label{sec: Main Result}
\begin{theorem}
Let $p \geq 2$ be a fixed integer. Suppose that the following conditions hold for $k \geq 2$.
\begin{enumerate}[\normalfont C1.]
\item The true probability distribution of $X \in \mathbb{R}^k$, $P_0$ has a probability density that is bounded on compact subsets of $\mathbb{R}^k$. 
\item The $\ell_1$-median of $P_0$ is unique.
\end{enumerate}
Then
\begin{enumerate}[label={\upshape(\roman*)}]
\item $\sqrt{n}(\hat{\theta}_n-\theta_0) \rightsquigarrow \mathrm{N}_k(0,\dot{\Psi}_0^{-1}\Sigma_0 \dot{\Psi}_0^{-1})$,
\item $\sqrt{n}(\theta(P)-\hat{\theta}_n) \rightsquigarrow \mathrm{N}_k(0,\dot{\Psi}_0^{-1}\Sigma_0 \dot{\Psi}_0^{-1})$ in $P_0$-probability.
\end{enumerate}
 Here $\dot{\Psi}_0=\int \dot{\psi}_{x,0}\ dP_0$, where
\begin{equation}
\dot{\psi}_{x,0}=\Big[\frac{\partial{\psi(x,\theta)}}{\partial{\theta}}\Big]_{\theta=\theta_0}.
\end{equation}
The matrix $\dot{\psi}_{x,0}$ looks like
\begin{equation}
\dot{\psi}_{x,0}=\frac{p-1}{\Vert x-\theta_0\Vert_p}\Big[\diag \left(\frac{\vert x_1-\theta_{01}\vert^{p-2}}{\Vert x-\theta_0\Vert_p^{p-2}},\dots,\frac{\vert x_k-\theta_{0k}\vert^{p-2}}{\Vert x-\theta_0\Vert_p^{p-2}}\right) - \frac{yy^T}{\Vert x-\theta_0\Vert_p^{2(p-1)}}\Big],
\end{equation}
with $y$ given by
\begin{equation}
y = \begin{bmatrix}
\vert x_1-\theta_{01} \vert^{p-1}\sign(x_1-\theta_{01}),\dots,\vert x_k-\theta_{0k} \vert^{p-1}\sign(x_k-\theta_{0k})
\end{bmatrix}^{T},
\end{equation}
and $$\Sigma_0=\frac{yy^T}{\Vert x-\theta_0\Vert_p^{2(p-1)}}.$$
Further if $k=2$, $(\mathrm{\romannum{1}})$ and $(\mathrm{\romannum{2}})$ hold for any $1<p<\infty$.
\end{theorem}
The uniqueness holds unless $P_0$ is completely supported on a straight line in $\mathbb{R}^k$, for $k\geq 2$ (Section 3, Chaudhuri \citep{chaudhuri1996geometric}). Finding the asymptotic distribution of $\sqrt{n}(\hat{\theta}_n-\theta_0)$ can be viewed as an application of the problem of finding the asymptotic distribution of a $Z$-estimator centered at its true value. The asymptotic theory of the $Z$-estimators has been studied extensively in the literature. Huber \citep{huber1967behavior} proved the asymptotic normality of $Z$-estimators when the parameter space is finite-dimensional. Van der Vaart \citep{van1995efficiency} extended Huber's \citep{huber1967behavior} theorem to the infinite-dimensional case. \par
Here we view $\theta(\mathbb{B}_n)$ as a bootstrapped version of the estimator $\hat{\theta}_n$, where the bootstrap weights are drawn from a $\mathrm{Dir}(n;1,1,\dots,1)$ distribution. In other words, $\theta(\mathbb{B}_n)$ satisfies the system of equations $\hat{\Psi}_n(\theta)=\mathbb{B}_n\psi(\cdot,\theta)=0$. Wellner and Zhan \citep{wellner1996bootstrapping} extended van der Vaart's \citep{van1995efficiency} infinite-dimensional $Z$-estimator theorem by showing that for any exchangeable vector of nonnegative bootstrap weights, the bootstrap analog of a $Z$-estimator conditional on the observations is also asymptotically normal. We use Wellner and Zhan's \citep{wellner1996bootstrapping} theorem to prove the asymptotic normality of $\theta(\mathbb{B}_n)$. Wellner and Zhan's \citep{wellner1996bootstrapping} theorem ensures that both $\sqrt{n}(\hat{\theta}_n-\theta_0)$, and $\sqrt{n}(\theta(P) - \hat{\theta}_n)$ given the data, converge in distribution to the same normal limit, thus proving Theorem 3.1. In the next subsection, we formally state Wellner and Zhan's \citep{wellner1996bootstrapping} theorem. In Section 7, we provide a detailed verification of the conditions of Wellner and Zhan's \citep{wellner1996bootstrapping} theorem in our situation.
\subsection{Bootstrapping a $Z$-estimator}
\label{sec: Bootstrapping}
Let $W_n = (W_{n1},W_{n2},\dots,W_{nn})$ be a set of bootstrap weights. The bootstrap empirical measure is defined as $\hat{\mathbb{P}}_n = n^{-1}\sum_{i=1}^nW_{ni}\delta_{X_i}$.
Wellner and Zhan \citep{wellner1996bootstrapping} assumed that the bootstrap weights $W= \{W_{ni},\ i=1,2,\dots,n,\ n=1,2,\dots \}$ form a triangular array defined on a probability space $(\mathfrak{Z},\mathscr{E},\hat{P})$. Thus $\hat{P}$ refers to the distribution of the bootstrap weights. According to Wellner and Zhan \citep{wellner1996bootstrapping}, the following conditions are imposed on the bootstrap weights:
\begin{enumerate}[(i)]
\item The vectors $W_n = (W_{n1},W_{n2},\dots,W_{nn})^T$ are exchangeable for every $n$, i.e., for any permutation $\pi=(\pi_1,\dots,\pi_n)$ of $\{1,2,\dots,n\}$, the joint distribution of $\pi(W_n)=(W_{n\pi_1},W_{n\pi_2},\dots,W_{n\pi_n})^T$ is same as that of $W_n$.
\item The weights $W_{ni} \geq 0$ for every $n,\ i$ and $\sum_{i=1}^nW_{ni}=n$ for all $n$.
\item The $L_{2,1}$ norm of $W_{n1}$ is uniformly bounded: for some $0<K<\infty$
\begin{equation}
\Vert W_{n1} \Vert_{2,1} = \int_0^{\infty}\sqrt{\hat{P}(W_{n1} \geq u)}\diff u \leq K.
\end{equation}
\item $\lim_{\lambda \rightarrow \infty}\limsup_{n \rightarrow \infty} \sup_{t \geq \lambda}(t^2\hat{P}\{W_{n1} \geq t)\}=0$.
\item $n^{-1}\sum_{i=1}^n(W_{ni}-1)^2 \rightarrow c^2>0$ in $\hat{P}$-probability for some constant $c>0$.
\end{enumerate}
Van der Vaart and Wellner \citep{van1996weak} noted that if $Y_1,\dots, Y_n$ are exponential random variables with mean 1, then the weights $W_{ni}=Y_i/\bar{Y}_n,\ i=1,\dots,n$, satisfy conditions (\romannumeral 1)--(\romannumeral 5), resulting in the Bayesian bootstrap scheme with $c=1$ because the left hand side in (\romannumeral 5) is given by $n^{-1}\sum_{i=1}^n(Y_i-\bar{Y}_n)^2/\bar{Y}_n^2\overset{P}\to\mathrm{Var}(Y)/\{\mathrm{E}(Y)\}^2=1$. To apply the bootstrap theorem, we also need to assume that the function class
\begin{equation}
\mathcal{F} = \mathcal{F}_R = \{\psi_j(\cdot,\theta):\Vert \theta - \theta_0 \Vert \leq R,\ j=1,2,\dots,k\}
\end{equation}
has \enquote{enough measurability} for randomization with independently and identically distributed multipliers to be possible and Fubini's theorem can be used freely. We call a function class $\mathcal{F} \in \mathfrak{m}(P)$ if $\mathcal{F}$ is countable, or if the empirical process $\mathbb{G}_n=\sqrt{n}(\mathbb{P}_n-P)$ is stochastically separable, or $\mathcal{F}$ is image admissible Suslin (see Chapter 5, Dudly \citep{dudley2014uniform} for a definition of the last). Now we formally state Wellner and Zhan's \citep{wellner1996bootstrapping} theorem for a sequence of consistent asymptotic bootstrap $Z$-estimators $\doublehat{\theta}_n$ of $\theta \in \mathbb{R}^k$, which satisfies the system of equations $\hat{\Psi}_n(\theta)=\hat{P}_n\psi(\cdot,\theta)=0$. 
\begin{theorem}[Wellner and Zhan] Assume that the class of functions $\mathcal{F} \in \mathfrak{m}(P_0)$ and the following conditions hold.
\begin{enumerate}[{\normalfont 1.}] 
\item There exists a $\theta_0 \equiv \theta(P_0)$ such that
\begin{equation}
\Psi(\theta_0)=P_0\psi(X,\theta_0)=0.
\end{equation}
The function $\Psi(\theta)=P_0\psi(X,\theta)$ is differentiable at $\theta_0$ with nonsingular derivative matrix $\dot{\Psi}_0$:
\begin{equation}
\dot{\Psi}_0=\Big[\frac{\partial{\Psi}}{\partial{\theta}}\Big]_{\theta=\theta_0}.
\end{equation}
\item
For any $\delta_n \rightarrow 0$, 
\begin{equation}
\sup \Big \{\frac{\Vert \mathbb{G}_n(\psi(\cdot,\theta)-\psi(\cdot,\theta_0))\Vert}{1+\sqrt{n}\Vert \theta - \theta_0 \Vert} : \Vert \theta - \theta_0 \Vert \leq \delta_n \Big\} = o_{P_0}(1).
\end{equation}
\item
The $k$-vector of functions $\psi$ is square-integrable at $\theta_0$ with covariance matrix
\begin{equation}
\Sigma_0 = P_0\psi(X,\theta_0)\psi^{T}(X,\theta_0) < \infty.
\end{equation}
For any $\delta_n \rightarrow 0$, the envelope functions 
\begin{equation}
D_n(x) = \sup \Big \{\frac{\vert \psi_j(x,\theta)-\psi_j(x,\theta_0)\vert}{1+\sqrt{n}\Vert \theta - \theta_0 \Vert} : \Vert \theta - \theta_0 \Vert \leq \delta_n,\  j=1,\ 2,\dots,k \Big\}
\end{equation}
satisfy 
\begin{equation}
\lim_{\lambda \rightarrow \infty}\limsup_{n \rightarrow \infty}\sup_{t \geq \lambda}t^2P_0(D_n(X_1)>t)=0.
\end{equation}
\item
The estimators $\doublehat{\theta}_n$ and $\hat{\theta}_n$ are consistent for $\theta_0$, i.e., $\Vert \hat{\theta}_n - \theta_0\Vert \overset{P_0}\to 0$ and $\Vert \doublehat{\theta}_n - \hat{\theta}_n \Vert \overset{\hat{P}}\to 0$ in $P_0$-probability.
\item
The bootstrap weights satisfy conditions \textup{(\romannumeral 1)}--\textup{(\romannumeral 5)}.
\end{enumerate}
Then
\begin{enumerate}[label={\upshape(\roman*)}]
\item $\sqrt{n}(\hat{\theta}_n-\theta_0) \rightsquigarrow \mathrm{N}_k(0,\dot{\Psi}_0^{-1}\Sigma_0\dot{\Psi}_0^{-1})$; 
\item $\sqrt{n}(\doublehat{\theta}_n-\hat{\theta}_n) \rightsquigarrow \mathrm{N}_k(0,c^2\dot{\Psi}_0^{-1}\Sigma_0\dot{\Psi}_0^{-1})$ in $P_0$- probability.
\end{enumerate}
\end{theorem}
It has already been mentioned that, for the Bayesian bootstrap weights, the value of the constant $c$ is $1$. Thus if $\psi(\cdot,\theta)$ defined in (2.2) satisfies the conditions in Theorem 3.2, Theorem 3.1 holds.
\section{Affine-equivariant Multivariate $\ell_1$-median}
\label{sec:Transformation}
Although the sample multivariate $\ell_1$-median is equivariant under location transformation and orthogonal transformation of the data, it is not equivariant under arbitrary affine transformation of the data. Chakraborty and Chaudhuri (\citep{chakraborty1996transformation}, \citep{chakraborty1998adaptive}) used a data-driven transformation-and-retransformation technique to convert the non-equivariant coordinatewise median to an affine-equivariant one. Chakraborty, Chaudhuri and Oja \citep{chakraborty1998operating} applied the same idea to the sample spatial median. \par
We use the transformation-and-retransformation technique to define an affine equivariant version of the multivariate $\ell_1$-median. Suppose that we have $n$ sample points $X_1,X_2,\dots,X_n \in \mathbb{R}^k$, with $n>k+1$. We consider the points $X_{i_0},X_{i_1},\dots,X_{i_k}$, where $\alpha =\{ i_0,i_1,\dots,i_k\}$ is a subset of $\{1,2,\dots,n\}$. The matrix $X(\alpha)$ consisting the columns $X_{i_1}-X_{i_0},X_{i_2}-X_{i_0},\dots, X_{i_k}-X_{i_0}$ is the data-driven transformation matrix. The transformed data points are $Z_j^{(\alpha)}= \{X(\alpha)\}^{-1}X_j$, $j \notin \alpha$. The matrix $X(\alpha)$ is invertible if $X_i,\ i=1,\dots,n,$ are independently and identically distributed samples from a distribution that is absolutely continuous with respect to the Lebesgue measure on $\mathbb{R}^k$. The sample $\ell_1$-median based on the transformed observations is then given by
\begin{equation}
\hat{\phi}_n^{(\alpha)}=\argmin_{\phi}\sum_{j \notin \alpha}\Vert Z_j^{(\alpha)}-\phi \Vert.
\end{equation}
We transform it back in terms of the original coordinate system as
\begin{equation}
\hat{\theta}_n^{(\alpha)}=X(\alpha)\hat{\phi}_n^{(\alpha)}.
\end{equation}
It can be shown that $\hat{\theta}_n^{(\alpha)}$ is affine equivariant. Chakraborty, Chaudhuri and Oja \citep{chakraborty1998operating} also noted that $X(\alpha)$ should be chosen in such a way that the matrix $\{X(\alpha)\}^T\Sigma^{-1}X(\alpha)$ is as close as possible to a matrix of the form $\lambda I_k$ where $\Sigma$ is the covariance matrix of $X$. The estimator $\hat{\theta}_n^{(\alpha)}$ does not have any population analog. Chakraborty, Chaudhui and Oja \citep{chakraborty1998operating} proved that conditional on $X(\alpha)$, the asymptotic distribution of the transformed-and-retransformed spatial median is normal.
\subsection{Bernstein-von Mises theorem for the affine-equivariant multivariate median}
Let $X_1,X_2,\dots,X_n \in \mathbb{R}^k$ be a random sample from a distribution $P$ that is absolutely continuous with respect to the Lebesgue measure on $\mathbb{R}^k$. Let $X(\alpha)$ be the transformation matrix and $Z_j^{(\alpha)}=\{X(\alpha)\}^{-1}X_j,\ j \notin \alpha$, be the transformed observations. The sample median of $X_1,\dots,X_n$ is denoted by $\hat{\theta}_n$.\par
Let the distribution of $Z_j^{(\alpha)},\ j \notin \alpha$, be denoted by $P_Z$. We equip $P_Z$ with a $\mathrm{DP}(\beta)$ prior. The true value of $P_Z$ is denoted by $P_{Z0}$, i.e., the distribution of $Z$ when $X \sim P_0$. Hence the Bayesian model can be described as 
\begin{equation}
Z_j^{(\alpha)}|P_{Z} \overset{iid}\sim P_Z,\quad P_Z \sim \mathrm{DP}(\beta),\ j \notin \alpha,
\end{equation}
which implies that
\begin{equation}
P_Z|\{Z_j^{(\alpha)}: j \notin \alpha\} \sim \mathrm{DP}(\beta+\sum_{j \notin \alpha}\delta_{Z_j}).
\end{equation}
Following the same arguments used in Section 2, it can be noted that since $P_Z\vert \{Z_j^{(\alpha)}:j \notin \alpha\} \sim \mathrm{DP}(\beta+\sum_{j \notin \alpha}\delta_{Z_j})$, $\Vert P_Z-\mathbb{B}_n \Vert_{\mathrm{TV}}=o_{\mathrm{Pr}}(n^{-1/2})$ a.s. $[P_{Z0}^{\infty}]$, with $\mathrm{Pr}=P_Z^{\infty}\times \mathbb{B}_n$. As in Theorem 3.1, the posterior distribution of $P_Z$ can be approximated by the Bayesian bootstrap process $\mathbb{B}_n$. Define
\begin{align}
\phi^{(\alpha)}(\mathbb{B}_n)&=\argmin_{\phi}\mathbb{B}_n\Vert Z^{(\alpha)}-\phi \Vert_p,\\
\phi^{(\alpha)}(P_Z) &= \argmin_{\phi}\big\{P_{Z}(\Vert Z^{(\alpha)}-\phi \Vert_p-\Vert Z^{(\alpha)}\Vert_p)\big\}.
\end{align}
Thus the transformed-and-retransformed medians are given by
\begin{equation}
\hat{\theta}_n^{(\alpha)}=X(\alpha) \hat{\phi}_n^{(\alpha)},\quad \theta^{(\alpha)}(\mathbb{B}_n)=X(\alpha)\phi^{(\alpha)}(\mathbb{B}_n).
\end{equation}
Also define $\theta^{(\alpha)}(P)=X(\alpha)\phi^{(\alpha)}(P_Z)$. We view $\hat{\phi}_n^{(\alpha)}$ as a $Z$-estimator satisfying $\Psi_{Z_n}(\phi)=\mathbb{P}_n\psi_Z(\cdot,\phi)=0$. The \enquote{population version} of $\Psi_{Z_n}(\phi)$ is denoted by $\Psi_Z(\phi)=P\psi_Z(\cdot,\phi)$. The real-valued elements of $\psi_Z(z,\phi)$ are then given by
\begin{equation}
\psi_{Z;j}(z,\phi) =\frac{\vert z_j-\phi_j \vert^{p-1}}{\Vert z-\phi\Vert_p^{p-1}}\sign(\phi_j-z_j), \quad j=1,\dots,k.
\end{equation}
Let $\phi_0^{(\alpha)}\equiv\phi^{(\alpha)}(P_{Z0})$ satisfy $\Psi_{Z0}(\phi^{(\alpha)})=P_{Z0}\psi_Z(\cdot,\phi^{(\alpha)})=0$. In the following, we denote $\dot{\Psi}^{(\alpha)}_{Z0}=\Big[\partial{\Psi}_{Z0}/{\partial{\phi}}\Big]_{\phi=\phi_0^{(\alpha)}}$
and $\Sigma^{(\alpha)}_{Z0}=P_{Z0}\psi_Z(\cdot,\phi_0^{(\alpha)})\psi_Z^{T}(\cdot,\phi_0^{(\alpha)})$. 
\begin{theorem}
Let $p \geq 2$ be a fixed integer. For $k \geq 2$ and a given subset $\alpha=\{i_0,i_1,\dots,i_k\}$ of $\{1,2,\dots,n\}$ with size $k+1$, suppose that the following conditions hold.
\begin{enumerate}[\normalfont C1.]
\item The true distribution of $Z^{(\alpha)}$, $P_{Z0}$ has a density which is bounded on compact subsets of $\mathbb{R}^k$.
\item The $\ell_1$-median of $P_{Z0}$ is unique.
\end{enumerate}
Then
\begin{enumerate}[label={\upshape(\roman*)}]
\item $\sqrt{n}(\hat{\theta}_n^{(\alpha)}-\theta^{(\alpha)}(P_0))\vert \{X_i:i \in \alpha\} \rightsquigarrow \mathrm{N}_k(0, X(\alpha)\{\dot{\Psi}^{(\alpha)}_{Z0}\}^{-1}\Sigma^{(\alpha)}_{Z0} \{\dot{\Psi}^{(\alpha)}_{Z0}\}^{-1}$ $\{X(\alpha)\}^T)$;
\item $\sqrt{n}(\theta^{(\alpha)}(P)-\hat{\theta}_n^{(\alpha)}) \rightsquigarrow \mathrm{N}_k(0,X(\alpha)\{\dot{\Psi}^{(\alpha)}_{Z0}\}^{-1}\Sigma^{(\alpha)}_{Z0} \{\dot{\Psi}^{(\alpha)}_{Z0}\}^{-1}\{X(\alpha)\}^T)$ in $P_0$-probability.
 Here $\dot{\Psi}^{(\alpha)}_{Z0}=\int \dot{\psi}_{Z,0}\ dP_{Z0}$, where
\begin{equation}
\dot{\psi}_{Z,0}=\Big[\frac{\partial{\psi_Z(z,\phi)}}{\partial{\phi}}\Big]_{\phi=\phi_0^{(\alpha)}}.
\end{equation}
The matrix $\dot{\psi}_{Z,0}$ is given by
\begin{align*}
\dot{\psi}_{Z,0}=\frac{p-1}{\Vert z-\phi_0^{(\alpha)}\Vert_p}\Big[\diag \left(\frac{\vert z_1-\phi_{01}^{(\alpha)}\vert^{p-2}}{\Vert z-\phi_0^{(\alpha)}\Vert_p^{p-2}},\dots,\frac{\vert z_k-\phi_{0k}^{(\alpha)}\vert^{p-2}}{\Vert z-\phi_0^{(\alpha)}\Vert_p^{p-2}}\right) -\\ \frac{yy^T}{\Vert z-\phi_0^{(\alpha)}\Vert_p^{2(p-1)}}\Big],
\end{align*}
with $y$ given by
\begin{equation}
y = \begin{bmatrix}
\vert z_1-\phi_{01}^{(\alpha)} \vert^{p-1}\sign(z_1-\phi_{01}^{(\alpha)}),\dots,\vert z_k-\phi_{0k}^{(\alpha)} \vert^{p-1}\sign(z_k-\phi_{0k}^{(\alpha)})
\end{bmatrix}^{T},
\end{equation}
and $$\Sigma_{Z0}^{(\alpha)}=\frac{yy^T}{\Vert z-\phi_{0k}^{(\alpha)}\Vert_p^{2(p-1)}}.$$
\end{enumerate}
Further if $k=2$, \textup{(\romannumeral 1)} and \textup{(\romannumeral 2)} hold for any $1<p<\infty$.
\end{theorem}
The uniqueness holds unless $P_{Z0}$ is completely supported on a straight line in $\mathbb{R}^k$, for $k \geq 2$, (Section 3, Chaudhuri \citep{chaudhuri1996geometric}). It can be noted that the $\mathrm{DP}(\beta)$ prior on $P_Z$ induces the $\mathrm{DP}(\beta \circ \psi^{-1})$ prior on $P \equiv P_Z \circ \psi^{-1}$, where $\psi(Y)=X(\alpha)Y$ with $Y \in \mathbb{R}^k$. Then the proof of the preceding theorem directly follows from Theorem 3.1. From Theorem 3.1, $\sqrt{n}(\hat{\phi}_n^{(\alpha)}-\phi_0^{(\alpha)})$ converges in distribution to $\mathrm{N}_k(0,\{\dot{\Psi}_{Z0}^{(\alpha)}\}^{-1}\Sigma_{Z0}^{(\alpha)} \{\dot{\Psi}_{Z0}^{(\alpha)}\}^{-1})$. Also $\sqrt{n}(\phi^{(\alpha)}(P)-\hat{\phi}_n^{(\alpha)})\vert \{Z_j^{(\alpha)}:j \notin \alpha \}$ converges conditionally in distribution to $\mathrm{N}_k(0,\{\dot{\Psi}_{Z0}^{(\alpha)}\}^{-1}\Sigma_{Z0} ^{(\alpha)}\{\dot{\Psi}_{Z0}^{(\alpha)}\}^{-1})$. Consequently, the conditional distribution of $\{\sqrt{n}(\hat{\theta}_n^{(\alpha)}-\theta^{(\alpha)}(P_0))\vert X_i:i \in \alpha\}$ converges to $\mathrm{N}_k(0,X(\alpha) \{\dot{\Psi}_{Z0}^{(\alpha)}\}^{-1}\Sigma^{(\alpha)}_{Z0}\{\dot{\Psi}_{Z0}^{(\alpha)}\}^{-1}$ $\{X(\alpha)\}^T)$. Also the conditional distribution of $\sqrt{n}(\theta^{(\alpha)}(P)-\hat{\theta}_n^{(\alpha)})$ given $X_1,\dots,$ $X_n$ converges to $\mathrm{N}_k(0,X(\alpha)\linebreak\{\dot{\Psi}_{Z
0}^{(\alpha)}\}^{-1}\Sigma^{(\alpha)}_{Z0}\{\dot{\Psi}_{Z0}^{(\alpha)}\}^{-1}\{X(\alpha)\}^T)$.
Apart from Theorem 3.1, this theorem uses the affine equivariance of the normal family: if a random vector $X \sim \mathrm{N}(\mu,\Sigma)$, then $Y=AX +b \sim \mathrm{N}(A\mu +b,\ A\Sigma A^T)$.\par
Though we are not interested in the case of $p=1$ (which reduces to coordinate-wise median), this may be noted that the conclusions of both Theorem 3.1 and 4.1 hold for that case as well.
\section{Bernstein-von Mises theorem for multivariate quantiles}
\label{BVM quantiles}
Let $X_i$, $i=1,\dots,n$, be independently and identically distributed observations from a $k$-variate distribution $P$ on $\mathbb{R}^k$ and $P$ is given the $\mathrm{DP}(\alpha)$ prior. The next theorem gives the joint posterior asymptotic distribution of the centered quantiles $\{\sqrt{n}(Q_P(u_1)-\hat{Q}_n(u_1)),\dots,\sqrt{n}(Q_P(u_m)-\hat{Q}_n(u_m))\}$ for $u_1,\dots,u_m \in B_q^{(k)}$.\par
Firstly we introduce some notations. For each $u$, the sample $u$-quantile is viewed as a $Z$-estimator that satisfies the system of equations $\Psi_n^{(u)}(\xi)=\linebreak \mathbb{P}_n\psi^{(u)}(\cdot,\xi)=0$. We denote the population version of $\Psi_n^{(u)}(\xi)$ by $\Psi^{(u)}(\xi)=P\psi^{(u)}(\cdot,\xi)$. The true value of $Q_P(u)$ is denoted by $Q_0{(u)}\equiv Q_{P_0}(u)$ and it satisfies the system of equations $\Psi_0^{(u)}(\cdot,\xi)=P_0\psi^{(u)}(\cdot,\xi)=0$. The real-valued components of $\psi^{(u)}(\cdot,\xi)$ are then given by
\begin{equation}
\psi^{(u)}_j(x,\xi) =\frac{\vert x_j-\xi_j \vert^{p-1}}{\Vert x-\xi\Vert_p^{p-1}}\sign(\xi_j-x_j)+u_j, \quad j=1,\dots,k.
\end{equation}
Define 
$\dot{\Psi}_0^{(u)}=\int \dot{\psi}_{x,0}^{(u)}\ dP_0$, where
\begin{equation}
\dot{\psi}_{x,0}^{(u)}=\Big[\frac{\partial{\psi^{(u)}(x,\xi)}}{\partial{\xi}}\Big]_{\xi=Q_0(u)}.
\end{equation}
The matrix $\dot{\psi}_{x,0}^{(u)}$ is given by
\begin{multline}
\dot{\psi}_{x,0}^{(u)}=\frac{p-1}{\Vert x-Q_0(u)\Vert_p}\Big[\diag \left( \frac{\vert x_1-Q_{01}(u)\vert^{p-2}}{\Vert x-Q_0(u)\Vert_p^{p-2}},\dots,\frac{\vert x_k-Q_{0k}(u)\vert^{p-2}}{\Vert x-Q_0(u)\Vert_p^{p-2}}\right) \\- \frac{yy^T}{\Vert x-Q_0(u)\Vert_p^{2(p-1)}}\Big],
\end{multline}
with $y$ given by
\begin{equation}
y = 
\big[\vert x_j-Q_{0j}(u) \vert^{p-1}\sign(x_j-Q_{0j}(u)):j=1,\dots,k\big]^T.
\end{equation}
In the above, $Q_{0j}(u),j=1,\dots,k$ denotes the $j$th component of the vector $Q_0(u)$. We also define $\Sigma_{0;u,v}=P_0\psi^{(u)}(x,Q_0(u))\{\psi^{(v)}(x,Q_0(v))\}^T$.
\begin{theorem}
Let $p \geq 2$ be a fixed integer. Suppose that the following conditions hold for $k \geq 2$.
\begin{enumerate}[\normalfont C1.]
\item The true distribution of $X$, $P_0$ has a density that is bounded on compact subsets of $\mathbb{R}^k$. 
\item For every $u_1,\dots,u_m \in  B_q^{(k)}$, the $u_1,\dots,u_m$-quantiles of $P_0$ are unique. 
\end{enumerate}
Then
\begin{enumerate}[label={\upshape(\roman*)}]
\item the joint distribution of $\big(\sqrt{n}(\hat{Q}_n(u_1)-Q_0(u_1)),\dots,\sqrt{n}(\hat{Q}_n(u_m)-Q_0(u_m))$ converges to a $km$-dimensional normal distribution with mean zero, and the $(j,l)$th block of the covariance matrix is given by $\{\dot{\Psi}_0^{(u_j)}\}^{-1}\Sigma_{0;u_j,u_l}\linebreak\{\dot{\Psi}_0^{(u_l)}\}^{-1}$, $1\leq j,\ l \leq m$;
\item given $X_1,\dots,X_n$, the posterior joint distribution of $\{\sqrt{n}(Q_P(u_1)-\hat{Q}_n(u_1)),\linebreak\dots,\sqrt{n}(Q_P(u_m)-\hat{Q}_n(u_m))\}$ converges to $km$-dimensional normal distribution with mean zero, and the $(j,l)$th block of the covariance matrix is given by $\{\dot{\Psi}_0^{(u_j)}\}^{-1}\Sigma_{0;u_j,u_l}\{\dot{\Psi}_0^{(u_l)}\}^{-1}$, $1\leq j,\ l \leq m$.
\end{enumerate}
Further if $k=2$, \textup{(\romannumeral 1)} and \textup{(\romannumeral 2)} hold for any $1<p<\infty$.
\end{theorem}
The uniqueness of the quantiles holds unless $P_0$ is completely supported on a straight line on $\mathbb{R}^k$, (Section 3, Chaudhuri \citep{chaudhuri1996geometric}). We give the proof of the previous theorem in Section 7.
\section{Simulation study and a real data application}
\label{sec:Simulation}
\subsection{Simulation study}
In this subsection, we demonstrate frequentist coverage of Bayesian credible sets for the multivariate $\ell_1$-median. We consider sample sizes $n=100$, $1000$ and $10000$. We generate data from multivariate normal and multivariate Laplace distributions with mean vector $\mu=0_k$ and covariance matrix $\Sigma=\mathrm{I}_k$, with cases $k=2$ and $k=3$.\par
To choose the transformation matrix for the transformation-and- retransformation technique, we follow the method discussed in Chakraborty, Chaudhuri and Oja \citep{chakraborty1998operating}. Recall that $X(\alpha)$ should be chosen in such a way that $\{X(\alpha)\}^T\Sigma ^ {-1}$ $X(\alpha)$ is as close as possible to a matrix $\lambda I_k$. We use the usual sample variance-covariance matrix (say $\hat{\Sigma}$) as an estimate of $\Sigma$. We want to make the eigenvalues of $X(\alpha)^T\hat{\Sigma} ^ {-1}X(\alpha)$ as equal as possible. For this we minimize the ratio between the arithmetic mean and the geometric mean of the eigenvalues of the positive definite matrix ${X(\alpha)}^T\hat{\Sigma} ^ {-1}X(\alpha)$, as suggested by Chakraborty, Chaudhuri and Oja \citep{chakraborty1998operating}.\par
We draw $N=2000$ many samples from the posterior distributions of $\theta(P)$ and $\theta^{(\alpha)}(P)$. We adopt two procedures for constructing the credible sets. In the first procedure, we construct the credible sets using $2.5$th and the $97.5$th coordinate-wise percentiles of the posterior samples. In the second procedure, we define $\bar{\theta}=N^{-1}\sum_{i=1}^N\theta_i$ and $\bar{S}=N^{-1}\sum_{i=1}^N(\theta_i-\bar{\theta})(\theta_i-\bar{\theta})^{\prime}$, where $\theta_i$,\ $i=1,\dots,N$, are the posterior samples. Then we compute $(\theta_i-\bar{\theta})^{\prime}{\bar{S}}^{-1}(\theta_i-\bar{\theta})$, $i=1,\dots,N$, and let $r$ be the $95$th percentile of them. Then the ellipsoid $\{\theta:(\theta-\bar{\theta})^{\prime}{\bar{S}}^{-1}(\theta-\bar{\theta}) \leq r\}$ is a $95\%$ credible set. \par
In the following tables, we show the sizes and coverages of the credible sets. It can be seen that as $n$ increases, the size decreases and the coverage increases. Also it can be noted that for $k=3$, the credible hyperrectangles are smaller than that for $k=2$.
\begin{table}[htbp]
\caption {Size and Coverage of $95\%$ Bayesian credible hyperrectangles of $\ell_1$-median with $k=2$} \label{tab:title} 
\begin{center}
 \begin{tabular}{|c|c|c|c|c|c|} 
 \hline
 \multirow{5}{5em}{Non Affine Equivariant} &  \multicolumn{5}{|c|}{Diameters of Credible hyperrectangles}\\\cline{2-6} 
 & Data & & $n=100$ & $n=1000$ & $n=10000$\\\cline{2-6}
&\multirow{2}{4em}{Normal}& $p=2$ & 0.681 & 0.185 & 0.063 \\\cline{3-6} 
& & $p=3$ & 0.587 & 0.192 & 0.059 \\\cline{2-6} 
&\multirow{2}{4em}{Laplace} & $p=2$ & 0.467 & 0.123 & 0.037 \\\cline{3-6} 
& & $p=3$ & 0.401 & 0.124 & 0.041 \\
\hline \hline
\multirow{4}{5em}{Affine Equivariant} & \multirow{2}{4em}{Normal} & p = 2 & 0.758 & 0.195 & 0.055 \\\cline{3-6}
& & $p=3$ & 0.599 & 0.158 & 0.057 \\\cline{2-6} 
&\multirow{2}{4em}{Laplace} & $p = 2$ & 0.673 & 0.187 & 0.061 \\\cline{3-6}
& & $p=3$ & 0.528 & 0. 185 & 0.056\\
\hline
\multirow{5}{5em}{Non Affine Equivariant} &  \multicolumn{5}{|c|}{Frequentist Coverage}\\\cline{2-6} 
&\multirow{2}{4em}{Normal}& $p=2$ & 0.937 & 0.942 & 0.960 \\\cline{3-6} 
& & $p=3$ & 0.940 & 0.945 & 0.955 \\\cline{2-6} 
&\multirow{2}{4em}{Laplace} & $p=2$ & 0.930 & 0.945 & 0.951 \\\cline{3-6} 
& & $p=3$ & 0.932 & 0.941 & 0.956\\
\hline \hline
\multirow{4}{5em}{Affine Equivariant} & \multirow{2}{4em}{Normal} & p =   2 & 0.921 & 0.939 & 0.949 \\\cline{3-6}
& & $p=3$ & 0.920 & 0.932 & 0.940 \\\cline{2-6} 
&\multirow{2}{4em}{Laplace} & $p = 2$ & 0.930 & 0.945 & 0.951 \\\cline{3-6}
& & $p=3$ & 0.932 & 0.941 & 0.956 \\
\hline
\end{tabular}
\end{center}
\end{table}
\begin{table}[htbp]
\caption {Size and Coverage of $95\%$ Bayesian credible hyperrectangles of $\ell_1$-median with $k=3$} 
\begin{center}
 \begin{tabular}{|c|c|c|c|c|c|} 
 \hline
 \multirow{5}{5em}{Non Affine Equivariant} &  \multicolumn{5}{|c|}{Diameters of Credible hyperrectangles}\\\cline{2-6} 
 & Data & & $n=100$ & $n=1000$ & $n=10000$\\\cline{2-6}
&\multirow{2}{4em}{Normal}& $p=2$ & 0.736 & 0.218 & 0.071 \\\cline{3-6} 
& & $p=3$ & 0.677 & 0.196 & 0.088\\\cline{2-6} 
&\multirow{2}{4em}{Laplace} & $p=2$ & 0.467 & 0.141 & 0.042\\\cline{3-6} 
& & $p=3$ & 0.402 & 0.113 & 0.039\\
\hline \hline
\multirow{4}{5em}{Affine Equivariant} & \multirow{2}{4em}{Normal} & p = 2 & 0.225 & 0.065 & 0.019  \\\cline{3-6}
& & $p=3$ & 0.287 & 0.055 & 0.008\\\cline{2-6} 
&\multirow{2}{4em}{Laplace} & $p = 2$ & 0.108 & 0.028 & 0.007 \\\cline{3-6}
& & $p=3$ & 0.188 & 0.037 & 0.010\\
\hline
\multirow{5}{5em}{Non Affine Equivariant} &  \multicolumn{5}{|c|}{Frequentist Coverage}\\\cline{2-6} 
&\multirow{2}{4em}{Normal}& $p = 2$ & 0.931 & 0.948 & 0.954 \\\cline{3-6} 
& & $p=3$ & 0.932 & 0.942 & 0.954 \\\cline{2-6} 
&\multirow{2}{4em}{Laplace} & $p = 2$ & 0.929 & 0.940 & 0.955 \\\cline{3-6} 
& & $p=3$ & 0.928 & 0.933 & 0.942\\
\hline \hline
\multirow{4}{5em}{Affine Equivariant} & \multirow{2}{4em}{Normal} & p = 2 & 0.928 & 0.951 & 0.959 \\\cline{3-6}
& & $p=3$ & 0.933 & 0.942 & 0.950 \\\cline{2-6} 
&\multirow{2}{4em}{Laplace} & $p = 2$ & 0.930 & 0.945 & 0.951 \\\cline{3-6}
& & $p=3$ & 0.928 & 0.933 & 0.950 \\
\hline
\end{tabular}
\end{center}
\end{table}

\begin{table}[htbp]
\caption {Size and Coverage of $95\%$ Bayesian credible ellipsoids of $\ell_1$-median with $k=2$} \label{tab:title} 
\begin{center}
 \begin{tabular}{|c|c|c|c|c|c|} 
 \hline
 \multirow{5}{5em}{Non Affine Equivariant} &  \multicolumn{5}{|c|}{Radii of credible ellipsoids}\\\cline{2-6} 
 & Data & & $n=100$ & $n=1000$ & $n=10000$\\\cline{2-6}
&\multirow{2}{4em}{Normal}& $p=2$ & 5.941 & 5.920 & 5.901 \\\cline{3-6} 
& & $p=3$ & 6.181 & 6.005 & 5.992 \\\cline{2-6} 
&\multirow{2}{4em}{Laplace} & $p=2$ & 6.113 & 5.995 & 5.889\\\cline{3-6} 
& & $p=3$ & 6.062 & 5.997 & 5.816\\
\hline \hline
\multirow{4}{5em}{Affine Equivariant} & \multirow{2}{4em}{Normal} & p = 2 & 5.960 & 5.880 & 5.799\\\cline{3-6}
& & $p=3$ & 6.005 & 5.991 & 5.895 \\\cline{2-6} 
&\multirow{2}{4em}{Laplace} & $p = 2$ & 6.110 & 6.001 & 5.992\\\cline{3-6}
& & $p=3$ & 6.129 & 5.993 &  5.881\\
\hline
\multirow{5}{5em}{Non Affine Equivariant} &  \multicolumn{5}{|c|}{Frequentist Coverage}\\\cline{2-6} 
&\multirow{2}{4em}{Normal}& $p=2$ & 0.950 & 0.953 & 0.960 \\\cline{3-6} 
& & $p=3$ & 0.942 & 0.954 & 0.956 \\\cline{2-6} 
&\multirow{2}{4em}{Laplace} & $p=2$ & 0.949 & 0.950 & 0.956 \\\cline{3-6} 
& & $p=3$ & 0.938 & 0.946 & 0.951\\
\hline \hline
\multirow{4}{5em}{Affine Equivariant} & \multirow{2}{4em}{Normal} & p =   2 & 0.942 & 0.955 & 0.959 \\\cline{3-6}
& & $p=3$ & 0.930 & 0.946 & 0.948 \\\cline{2-6} 
&\multirow{2}{4em}{Laplace} & $p = 2$ & 0.930 & 0.948 & 0.955 \\\cline{3-6}
& & $p=3$ & 0.944 & 0.949 & 0.955 \\
\hline
\end{tabular}
\end{center}
\end{table}
\begin{table}[htbp]
\caption {Size and Coverage of $95\%$ Bayesian credible ellipsoids of $\ell_1$-median with $k=3$} \label{tab:title} 
\begin{center}
 \begin{tabular}{|c|c|c|c|c|c|} 
 \hline
 \multirow{5}{5em}{Non Affine Equivariant} &  \multicolumn{5}{|c|}{Radii of credible ellipsoids}\\\cline{2-6} 
 & Data & & $n=100$ & $n=1000$ & $n=10000$\\\cline{2-6}
&\multirow{2}{4em}{Normal}& $p=2$ & 7.723 & 7.701 & 7.671 \\\cline{3-6} 
& & $p=3$ & 7.811 & 7.768 & 7.682 \\\cline{2-6} 
&\multirow{2}{4em}{Laplace} & $p=2$ & 7.868 & 7.716 & 7.279\\\cline{3-6} 
& & $p=3$ & 8.017 & 7.992 & 7.810\\
\hline \hline
\multirow{4}{5em}{Affine Equivariant} & \multirow{2}{4em}{Normal} & p = 2 & 8.005 & 7.997 & 7.814\\\cline{3-6}
& & $p=3$ & 8.112 & 8.001 & 7.902 \\\cline{2-6} 
&\multirow{2}{4em}{Laplace} & $p = 2$ & 8.103 & 7.995 & 7.877\\\cline{3-6}
& & $p=3$ & 8.106 & 7.968 &  7.852\\
\hline
\multirow{5}{5em}{Non Affine Equivariant} &  \multicolumn{5}{|c|}{Frequentist Coverage}\\\cline{2-6} 
&\multirow{2}{4em}{Normal}& $p = 2$ & 0.933 & 0.953 & 0.958 \\\cline{3-6} 
& & $p=3$ & 0.935 & 0.941 & 0.955 \\\cline{2-6} 
&\multirow{2}{4em}{Laplace} & $p = 2$ & 0.939 & 0.943 & 0.955 \\\cline{3-6} 
& & $p=3$ & 0.948 & 0.953 & 0.956\\
\hline \hline
\multirow{4}{5em}{Affine Equivariant} & \multirow{2}{4em}{Normal} & p = 2 & 0.934 & 0.950 & 0.949 \\\cline{3-6}
& & $p=3$ & 0.933 & 0.942 & 0.950 \\\cline{2-6} 
&\multirow{2}{4em}{Laplace} & $p = 2$ & 0.935 & 0.946 & 0.953 \\\cline{3-6}
& & $p=3$ & 0.938 & 0.943 & 0.952 \\
\hline
\end{tabular}
\end{center}
\end{table}
\newpage
\subsection{Analysis of Fisher's iris data}
Fisher's iris data consists of three plant species, namely, Setosa, Virginica and Versicolor and four features, namely, sepal length, sepal width, petal length and petal width measured for each sample. We compute the 95\% credible ellipsoid of the 4-dimensional multivariate $\ell_1$-median with $p=2$ and report its four principal axes in Table 6.5. 
\begin{table}[h]
\caption{Principal axes of 95\% credible ellipsoid of spatial median}
\begin{center}
\begin{tabular}{|c|c|c|c|}
\hline
1st axis & 2nd axis & 3rd axis & 4th axis\\
\hline
0.0580 & $\minus 0.3129$ & $\minus 0.6747$ & $\minus 0.6629$\\
\hline
$\minus 0.1461$ & 0.2193 & $\minus 0.6143$ & $\minus 0.7437$\\
\hline
$\minus 0.2965$ & 0.8626 & 0.4089 & $\minus 0.0252$\\
\hline
0.9420 & 0.3252 & $\minus 0.0081$ & $\minus 0.0824$\\
\hline
\end{tabular}
\end{center}
\end{table}
Also in Figures 1, 2 and 3, we plot 6 pairs of features for each species and the credible ellipsoids for the corresponding two dimensional $\ell_1$-medians with $p=2$.
\begin{figure}[h]
\centering
\caption{$95\%$ Credible ellipsoids for the species Setosa}
\includegraphics[trim={0 5cm 0 6.8cm},width=0.95\textwidth,clip]{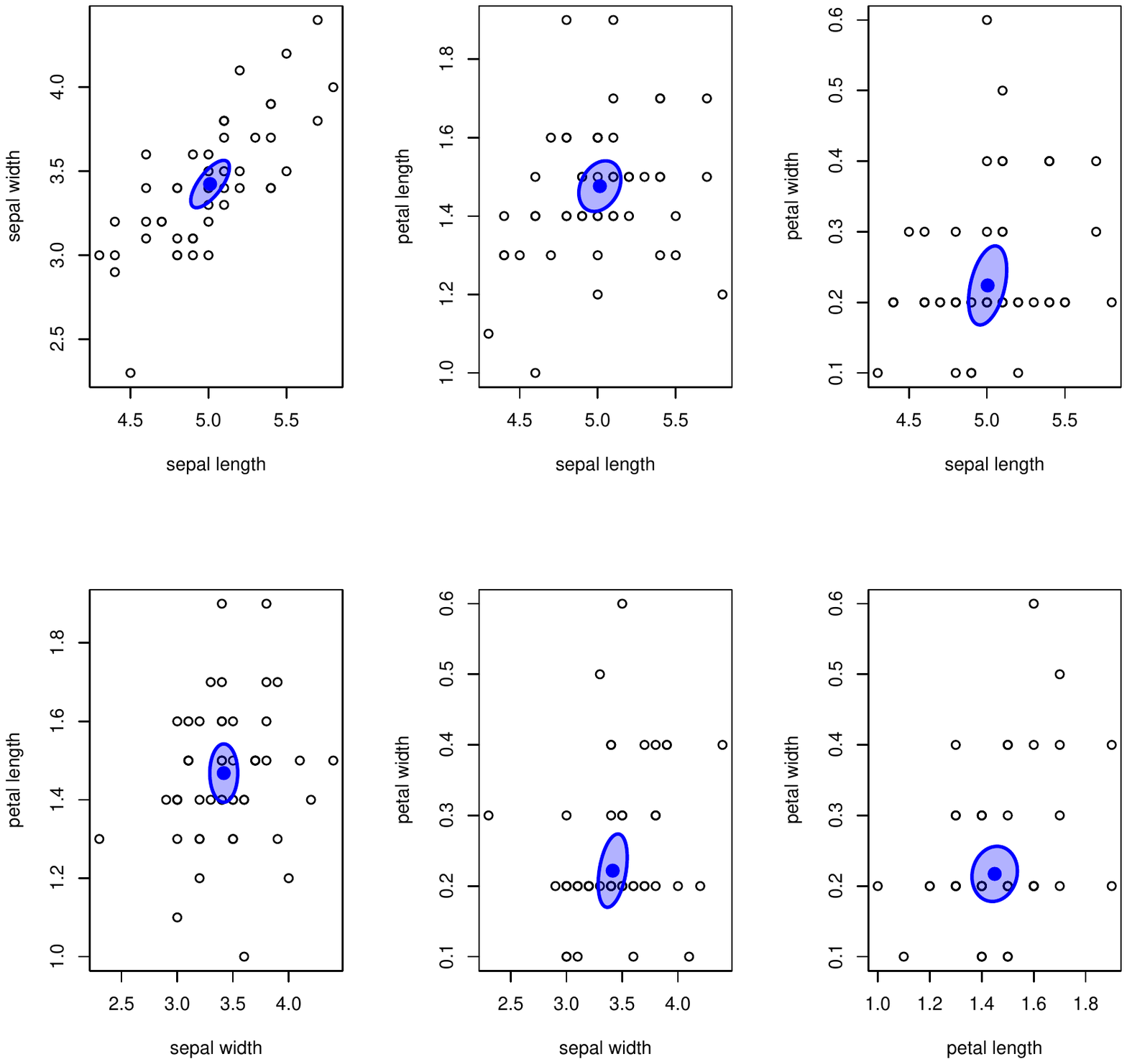}
\end{figure}
\begin{figure}[h]
\centering
\caption{$95\%$ Credible ellipsoids for the species Virginica}
\includegraphics[trim={0 5cm 0 6.8cm},width=0.95\textwidth,clip]{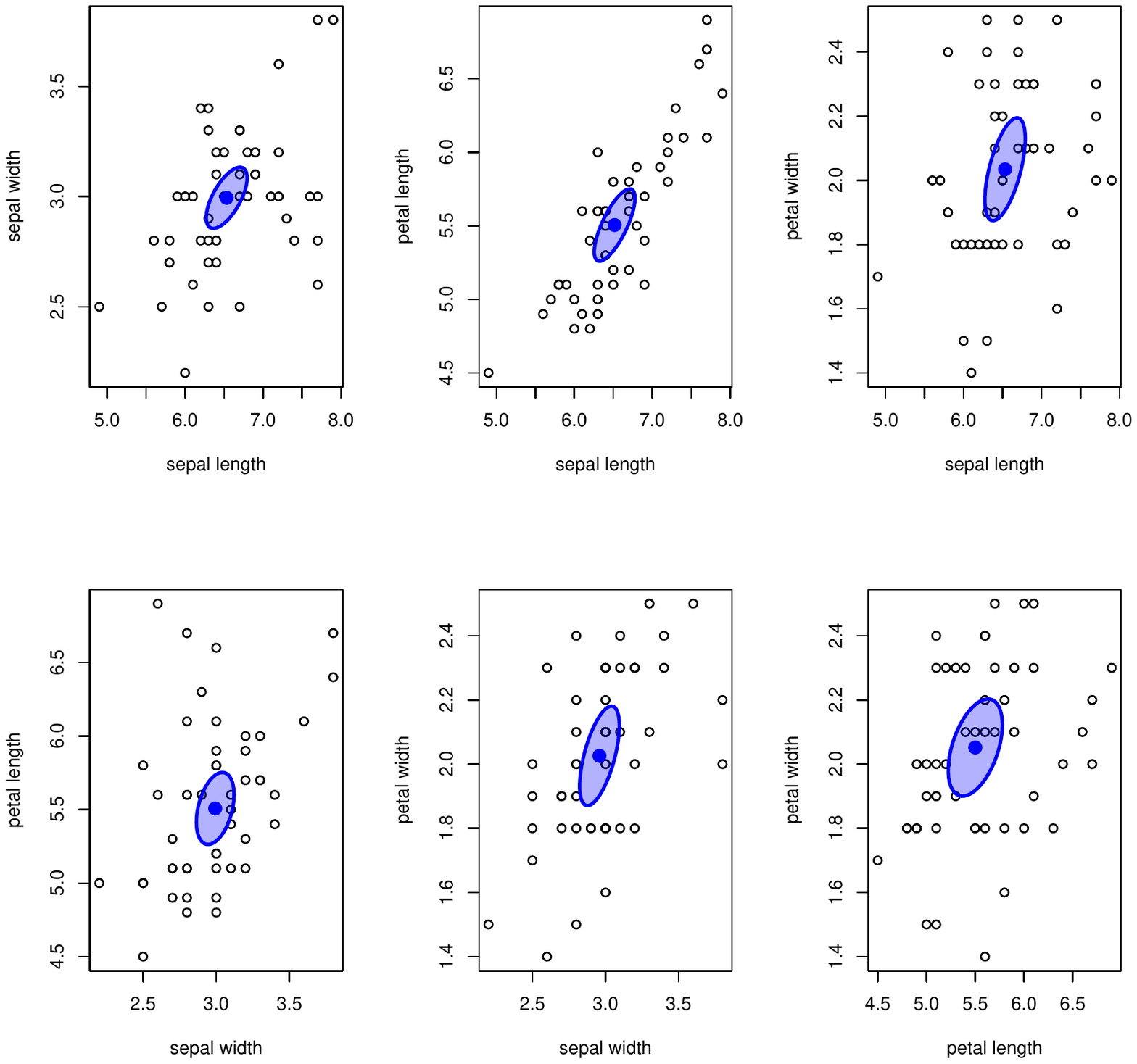}
\end{figure}
\begin{figure}[h]
\centering
\caption{$95\%$ Credible ellipsoids for the species Versicolor}
\includegraphics[trim={0 5cm 0 6.8cm},width=0.95\textwidth,clip]{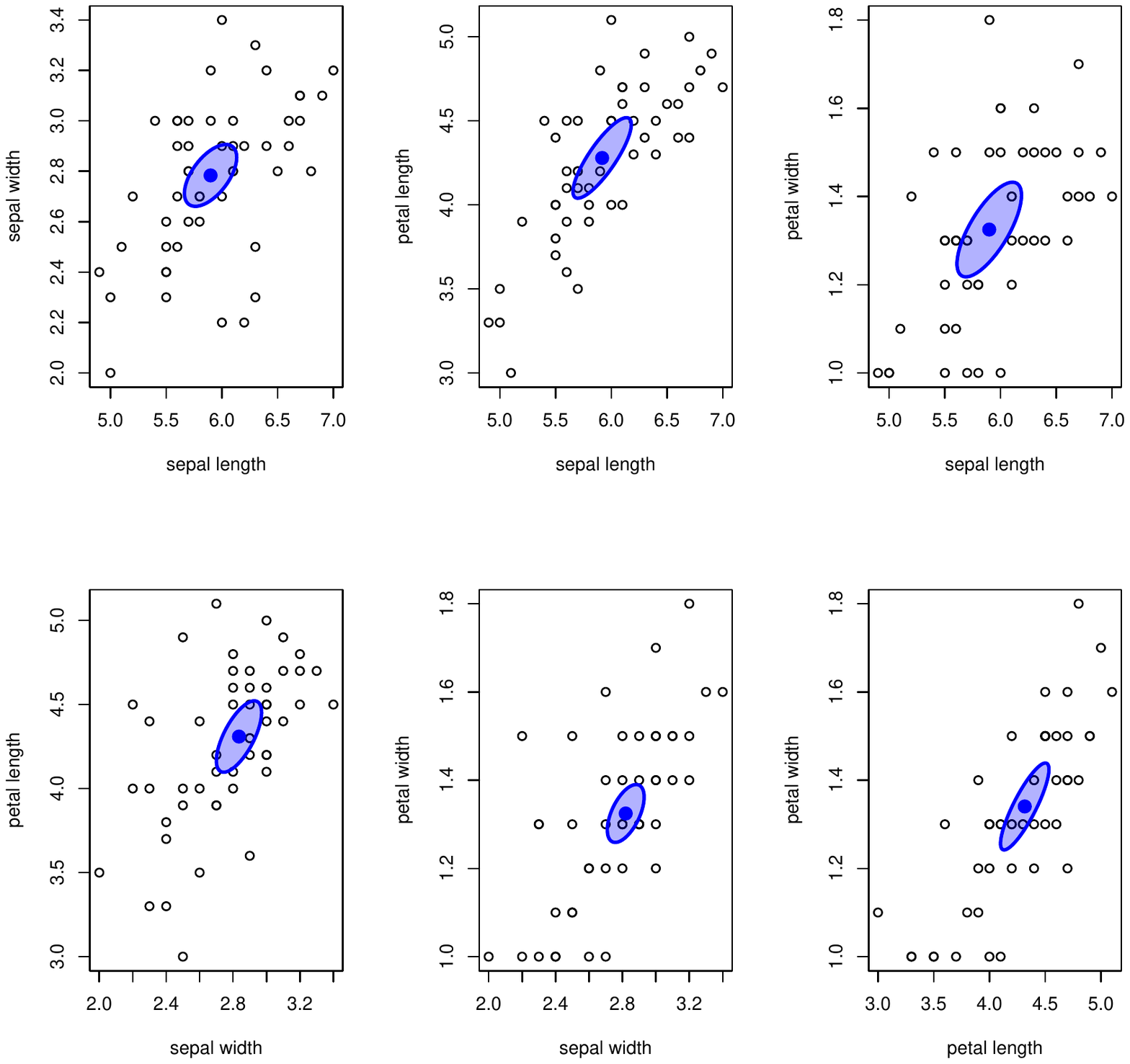}
\end{figure}
\section{Proof}
\label{sec: Proof}
\subsection{Technical preliminaries}
Before we proceed to the proof, we introduce some notations and definitions that we will need in the proof. For the rest of the paper, $L_r(Q)$ denotes the norm $\Vert f \Vert_{Q,r}=\left(\int \vert f \vert ^rdQ\right)^{1/r}$.
\begin{definition}[Covering Numbers and Uniform Entropy]
The covering number $N(\epsilon,\mathcal{F},\Vert \cdot \Vert)$ is the minimal number of balls $\{g:\Vert g-f \Vert < \epsilon\}$ of radius $\epsilon$ needed to cover $\mathcal{F}$. \par 
A class of functions $\mathcal{F}$ with the envelope function $F$ is said to satisfy the uniform entropy condition if
\begin{equation}
\int_0^{\infty}\sup_{Q}\sqrt{\log N(\epsilon \Vert F \Vert_{Q,2},\mathcal{F},L_2(Q))}\diff\epsilon < \infty,
\end{equation}
where the supremum has been taken over all finite discrete probability measures on with $\Vert F \Vert^2_{Q,2}= \int F^2\mathrm{d}Q>0$.
\end{definition}
\begin{definition}[Bracketing Numbers]
For two functions $l$ and $u$, the bracket $[l,u]$ is defined to be the set of all functions $f$ with $l \leq f \leq u$. An $\epsilon$-bracket in $L_r(P)$ is a bracket $[l,u]$ with $\Vert u-l \Vert \leq \epsilon$. \par 
The bracketing number $N_{[\,]}(\epsilon,\mathcal{F},L_r(P))$ is the minimum number of $\epsilon$-brackets needed to cover $\mathcal{F}$.
\end{definition}
\begin{definition}[VC Class of Sets]
Let $\mathcal{C}$ be a collection of subsets of a set $\mathfrak{X}$. We say that an arbitrary subset $S=\{x_1,x_2,\dots,x_n\}$ of $\mathcal{X}$ is shattered by $\mathcal{C}$ if for every subset $S^{\prime}\subseteq S$, there exists $C \in \mathcal{C}$ such that $S^{\prime}=S \cap C$.\par 
The VC-index of the class $\mathcal{C}$ is the smallest $n$ for which no set of size $n$ is shattered by $\mathcal{C}$ i.e. 
\begin{equation}
V(\mathcal{C}) = \inf\{n:\max_{x_1,\dots,x_n} \Delta_n(\mathcal{C},x_1,\dots,x_n) < 2^n\},
\end{equation}
where $\Delta_n(\mathcal{C},x_1,\dots,x_n) = \#\{C \cap \{x_1,\dots,x_n\}: C \in \mathcal{C}\}$. A collection of measurable sets is called a VC class of sets if its VC-index is finite.
\end{definition}
\begin{definition}[Glivenko-Cantelli Class]
A function class $\mathcal{F}$ for which $$\Vert \mathbb{P}_n-P\Vert_{\mathcal{F}}=\sup_{f \in \mathcal{F}}\vert \mathbb{P}_nf-Pf \vert \rightarrow 0,$$ is called a $P$-Glivenko-Cantelli class, where the convergence can be in probability or almost surely.
\end{definition}
\begin{definition}[Donsker Class]
For a function class $\mathcal{F}$ and a probability measure $P$, suppose that 
\begin{equation}
\sup_{f \in \mathcal{F}} \vert f(x)-Pf \vert < \infty.
\end{equation}
Let $\mathcal{L}^{\infty}(T)$ be the set of all functions $f:T \mapsto R$ such that $$\sup_{t \in T}\vert f(t) \vert < \infty.$$ By viewing the empirical process $\{\mathbb{G}_nf: f \in \mathcal{F}\}$ as a map into $\mathcal{L}^{\infty}(\mathcal{F})$, if
\begin{equation}
\mathbb{G}_n = \sqrt{n}(\mathbb{P}_n-P) \rightsquigarrow \mathbb{G}, \qquad \mathrm{in} \ \mathcal{L}^{\infty}(\mathcal{F}),
\end{equation}
for a tight Borel measurable element $\mathbb{G}$ $\mathrm{in}$ $\mathcal{L}^{\infty}(\mathcal{F})$, then we say that $\mathcal{F}$ is called a $P$-Donsker class.
\end{definition}
\subsection{Proof of Theorem 3.1}
We give the proof in two steps. In the first step, we verify the conditions of Theorem 3.2 in our situation and show that the asymptotic conditional distribution of $\sqrt{n}(\theta(\mathbb{B}_n)-\hat{\theta}_n)$ is $\mathrm{N}_k(0,\dot{\Psi}_0^{-1}\Sigma_0\dot{\Psi}_0^{-1})$. In the second step (Lemma 7.1), we show that the asymptotic posterior distribution of $\sqrt{n}(\theta(P)-\hat{\theta}_n)$ is the same as the asymptotic conditional distribution of $\sqrt{n}(\theta(\mathbb{B}_n)-\hat{\theta}_n)$.\par
We show that $\psi(\cdot,\theta)$ defined in (2.2) satisfies the conditions in Theorem 3.2. Firstly, we need to show that the function class $\mathcal{F}_R \in \mathfrak{m}(P_0)$ where $\mathcal{F}_R$ is defined in (3.5). To achieve this, we prove that the empirical process $\mathbb{G}_n=\sqrt{n}(\mathbb{P}_n-P_0)$ indexed by $\mathcal{F}_R$ is stochastically separable. It can be noted that $\psi_j(x,\theta),\ j=1,\dots,k$, are left-continuous at each $x$ for every $\theta$ such that $\Vert \theta-\theta_0 \Vert \leq R$. Hence there exists a null set $N$ and a countable $\mathcal{G} \subset \mathcal{F}_R$ such that, for every $\omega \notin N$ and $f \in \mathcal{F}_R$, we have a sequence $g_m \in \mathcal{G}$ with $g_m \rightarrow f$ and $\mathbb{G}_n(g_m,\omega) \rightarrow \mathbb{G}_n(f,\omega)$. For more details, see Chapter 2.3, van der Vaart and Wellner \citep{van1996weak}.
\begin{proof}[Verification of Condition \textup{1} in Theorem \textup{3.2}]
By Condition C2 in Theorem 3.1, the $\ell_1$-median of $P_0$ exists and is unique. Hence there exists a $\theta_0 \equiv \theta(P_0) \in \mathbb{R}^k$ such that (3.6) is satisfied. Also, $\Psi_0(\theta)=P_0\psi(X,\theta)$ is differentiable from Condition C1. This follows from the fact that for a fixed $\theta \in \mathbb{R}^k$ and a density $f$ bounded on compact subsets of $\mathbb{R}^k$, $P_0(\Vert X-\theta \Vert ^{-1})$ is finite. This can be verified by using $k$-dimensional polar transformation for which the determinant of the Jacobian matrix contains $(k-1)$th power of the radius vector (Chaudhuri \citep{chaudhuri1996geometric}).
\end{proof}
\begin{proof}[Verification of Condition \textup{2} in Theorem \textup{3.2}]
From Wellner and Zhan \citep{wellner1996bootstrapping}, Condition 2 is satisfied if, $\mathcal{F}_R$ in (3.5) is $P_0$-Donsker for some $R>0$ and
\begin{equation}
\max_{1 \leq j \leq k} P_0(\psi_j(\cdot,\theta)-\psi_j(\cdot,\theta_0))^2 \rightarrow 0,
\end{equation}
as $\theta \rightarrow \theta_0$.
In order to prove that $\mathcal{F}_R$ is $P_0$-Donsker, we define the following two function classes:
\begin{align}
\mathcal{F}_{1R} &= \Big \{\frac{\vert x_j-\theta_j \vert ^{p-1} }{\Vert x-\theta \Vert_p^{p-1}}: j=1,2,\dots,k,\Vert \theta - \theta_0 \Vert \leq R \Big \}, \\
\mathcal{F}_{2R} &= \Big \{\sign(\theta_j-x_j): j=1,2,\dots,k,\Vert \theta - \theta_0 \Vert \leq R \Big \}.
\end{align}
From Example 2.10.23 of van der Vaart and Wellner \citep{van1996weak}, if $\mathcal{F}_{1R}$ and $\mathcal{F}_{2R}$ satisfy the uniform entropy condition and are suitably measurable, then $\mathcal{F}_R= \mathcal{F}_{1R}\mathcal{F}_{2R}$ is $P_0$-Donsker provided their envelopes $F_{1R}$ and $F_{2R}$ satisfy $P_0F_{1R}^2F_{2R}^2 < \infty$.\par 
Every $f \in \mathcal{F}_{1R}$ is continuous at each $x$, hence $\mathcal{F}_{1R}$ has a countable subset $\mathcal{G}$ such that for every $f \in \mathcal{F}_{1R}$ there exists a sequence $g_m(x) \rightarrow f(x)$ for every $x$. Then by Example 2.3.4. of van der Vaart and Wellner \citep{van1996weak}, $\mathcal{F}_{1R}$ is $P$-measurable for every $P$. Since every $f \in \mathcal{F}_{2R}$ is left-continuous at each $x$, same conclusion holds for $\mathcal{F}_{2R}$ as well.\par
A class of functions $\mathcal{F}$ is called a VC-major class of functions if the sets $\{x: f(x)>t\}$ with $f$ ranging over $\mathcal{F}$ and $t$ over $\mathbb{R}$ form a VC-class of sets. By Corollary 2.6.12 of van der Vaart and Wellner \citep{van1996weak}, if $\mathcal{F}_{1R}$ is a bounded VC-major class of functions, then it satisfies the uniform entropy condition. It is easy to see that $\mathcal{F}_{1R}$ is bounded.
We now show that $\mathcal{F}_{1R}$ is a VC-major class of sets, that is, the sets $\{x:f(x)>t\}$ with $f$ varying over $\mathcal{F}_{1R}$ and $t$ over $\mathbb{R}$ form a VC class of sets. Define the collection of sets
$\mathcal{S}=\{S_{\theta,t}: \Vert \theta-\theta_0 \Vert \leq R,\ t \in \mathbb{R}\}$,
where $S_{\theta,t}$ is defined as
\begin{equation}
S_{\theta,t}=\Big \{x:\frac{\vert x_j-\theta_j \vert ^{p-1} }{\Vert x-\theta \Vert_p^{p-1}} > t,\ j=1,2,\dots,k\Big \}.
\end{equation}
We need to show $\mathcal{S}$ is a VC-class of sets. Note that $S_{\theta,t}=\cap_{j=1}^k S_{\theta,t}^j$, where $S_{\theta,t}^j$ is defined as
\begin{equation}
S_{\theta,t}^j = \Big \{x:\frac{\vert x_j-\theta_j \vert ^{p-1} }{\Vert x-\theta \Vert_p^{p-1}} > t\Big \}.
\end{equation}
In view of Lemma 2.6.17 of van der Vaart and Wellner \citep{van1996weak}, it is enough to show
\begin{equation}
\mathcal{S}^j=\{S_{\theta,t}^j: \Vert \theta - \theta_0 \Vert \leq R,\ t \in \mathbb{R}\}
\end{equation}
is a VC-class for every $j$; because if every $\mathcal{S}^j$ is a VC-class of sets, $\mathcal{S}=\sqcap_{j=1}^k \mathcal{S}^j=\{\cap_{j=1}^kS^j:S^j \in \mathcal{S}^j\}$ is also a VC-class of sets. Hence we only show that $\mathcal{S}^1=\{S_{\theta,t}^1: \Vert \theta - \theta_0 \Vert \leq R,\ t \in \mathbb{R}\}$ is a VC-class of sets. We can write $S_{\theta,t}^1$ as
\begin{align}
S_{\theta,t}^1 
&= \Big \{x:\vert x_1-\theta_1 \vert^p > t^{p/(p-1)}\sum_{j=1}^k \vert x_j-\theta_j \vert^p)\Big \}\nonumber \\
&= \Big \{x:\vert x_1-\theta_1 \vert^p >\frac{t^{p/(p-1)}}{1-t^{p/(p-1)}}\sum_{j=2}^k \vert x_j-\theta_j \vert^p)\Big \}. \nonumber
\end{align}
Define $R_{\theta,c}^1 = \Big \{x:\vert x_1-\theta_1 \vert^p >c\sum_{j=2}^k \vert x_j-\theta_j \vert^p\Big \}$ and $\mathcal{R}^1 = \{R_{\theta,c}^1: \theta \in \mathbb{R}^k, c \in \mathbb{R}\}$.
It is enough to show that $\mathcal{R}^1$ is a VC-class, since $\mathcal{R}^1$ contains $\mathcal{S}^1$. Define
\begin{align*}
A_{0;\theta,c} =& \{x:(x_1-\theta_1)^p>c\sum_{j=2}^k(x_j-\theta_j)^p\}\bigcap_{j-1}^k\{x:\ x_j-\theta_j \geq 0\},\\
A_{0;\theta,c}^{\prime} =& \{x:(\theta_1-x_1)^p>c\sum_{j=2}^k(x_j-\theta_j)^p,\ \theta_1-x_1 \leq 0\}\bigcap_{j=2}^k\{ x:x_j-\theta_j \geq 0\}.
\end{align*}
For $i,\ j=2,\dots,k,\ i \neq j$, we define
\begin{align*}
A_{i;\theta,c} =&\{x:(x_1-\theta_1)^p>c(\theta_i-x_i)^p + c\sum_{\substack {j=2 \\ j \neq i}}^k (x_j-\theta_j)^p,\\ & x_1-\theta_1 \geq 0,\ \theta_i-x_i \geq 0,\ x_j-\theta_j \geq 0,\ j=2,\dots,k,\ j \neq i\},\\
A_{i;\theta,c}^{\prime} =&\{x:(\theta_1-x_1)^p>c(\theta_i-x_i)^p + c\sum_{\substack {j=2 \\ j \neq i}}^k (x_j-\theta_j)^p,\\ & \theta_1-x_1 \geq 0,\ \theta_i-x_i \geq 0,\ x_j-\theta_j \geq 0,\ j=2,\dots,k,\ j \neq i\},\\
A_{ij;\theta,c}=& \{x:(x_1-\theta_1)^p>c(\theta_i-x_i)^p + c(\theta_j-x_j)^p + c\sum_{\substack {l=2 \\ l \neq i,j}}^k (x_l-\theta_l)^p,\\ & x_1-\theta_1 \geq 0,\ \theta_i-x_i \geq 0,\ \theta_j-x_j \geq 0,\ x_l-\theta_l \geq 0,\ l=2,\dots,k,\\&\ l \neq i,j\},\\
A_{ij;\theta,c}^{\prime}=& \{x:(\theta_1-x_1)^p>c(\theta_i-x_i)^p + c(\theta_j-x_j)^p + c\sum_{\substack {l=2 \\ l \neq i,j}}^k (x_l-\theta_l)^p,\\ & \theta_1-x_1 \geq 0,\ \theta_i-x_i \geq 0,\ \theta_j-x_j \geq 0,\ x_l-\theta_l \geq 0,\ l=2,\dots,k,\\&\ l \neq i, j\}.
\end{align*}
Continuing this pattern, finally
\begin{align*}
A_{n-1;\theta,c}=& \{x:(x_1-\theta_1)^p> c\sum_{j=2}^k(\theta_j-x_j)^p,\ x_1-\theta_1 \geq 0,\ \theta_j-x_j \geq 0,\\& j=2,\dots,k\},\\
A_{n-1;\theta,c}^{\prime}=& \{x:(\theta_1-x_1)^p> c\sum_{j=2}^k(\theta_j-x_j)^p,\ \theta_1-x_1 \geq 0,\ \theta_j-x_j \geq 0,\\& j=2,\dots,k\}.
\end{align*}
Using the preceding notations, $R_{\theta,c}^1$ can be written as
\begin{align}
R_{\theta,c}^1 =& A_{0;\theta,c}\cup A^{\prime}_{0;\theta,c}\cup\{\bigcup_{l=1}^{n-2}B_{l;\theta,c}\}\cup \{\bigcup_{l=1}^{n-2}B^{\prime}_{l;\theta,c}\}\cup A_{n-1;\theta,c}\cup A_{n-1;\theta,c}^{\prime},
\end{align}
where
\begin{align*}
B_{1;\theta,c}=\bigcup_{i=2}^k A_{i;\theta,c};\ B_{1;\theta,c}^{\prime}=\bigcup_{i=2}^k A^{\prime}_{i;\theta,c},\ B_{2;\theta,c}=\mathop{\bigcup_{i=2}^k\bigcup_{j=2}^k}_{i<j}A_{ij;\theta,c},
\end{align*}
and so on. Since all the sets on the right hand side of (7.11) are in the same form, if $\mathcal{C}=\{A_{0;\theta,c}:\theta \in \mathbb{R}^k,c \in \mathbb{R}\}$ forms a VC-class of sets, then $\mathcal{R}^1$ also forms a VC-class of sets. This follows from Lemma 2.6.17 of van der Vaart and Wellner \citep{van1996weak}, which says that if $\mathcal{F}$ and $\mathcal{G}$ are VC-classes, then $\mathcal{F}  \sqcup \mathcal{G}=\{F \cup G: F \in \mathcal{F},\ G \in \mathcal{G}\}$ also forms a VC-class. We can write $A_{0;\theta,c}$ as
\begin{equation}
A_{0;\theta,c} = \{x:(x_1-\theta_1)^p>c\sum_{j=2}^k(x_j-\theta_j)^p\}
\bigcap_{j=1}^k\{x:x_j-\theta_j \geq 0\}.
\end{equation}
Since $p$ is a positive integer greater than $1$, we can write
\begin{eqnarray*}
\lefteqn{\{x:(x_1-\theta_1)^p>c\sum_{j=2}^k(x_j-\theta_j)^p\}} \\
&=\{x: \sum_{r=0}^p {p \choose r}x_1^{p-r}(-1)^r\theta_1^r > c\sum_{j=2}^k\sum_{r=0}^p {p \choose r}x_j^{p-r}(-1)^r\theta_j^r\}\\
&=\{x:x_1^p-px_1^{p-1}\theta_1+ \dots + (-1)^p\theta_1^p-c\sum_{j=2}^k(x_j^p-px_j^{p-1}\theta_j\\
 &+ \dots + (-1)^p\theta_j^p) >0\}.
\end{eqnarray*}
Consider the map $x \mapsto \phi(x)$, where
\begin{equation}
\phi(x)=\big\{x_1^{p-1},x_1^{p-2},\dots,x_1,\sum_{j=2}^kx_j^p,\sum_{j=2}^kx_j^{p-1},\dots,\sum_{j=2}^kx_j,1\big\}.
\end{equation}
Note that the class of functions $\{g_a(x)=a^T\phi(x)+x_1^p: a \in \mathbb{R}^{2p}\}$
is a finite dimensional vector space. The collection of sets $$\{x:(x_1-\theta_1)^p>c\sum_{j=2}^k(x_j-\theta_j)^p,\ \theta \in \mathbb{R}^k,\ c \in \mathbb{R}\}$$ is the same as $\mathcal{C}_1=\{x:g_a(x)>0,\ a \in \mathbb{R}^{2p}\}$ and $\mathcal{C}_1$ is a VC-class of sets by Lemma 2.6.15 of van der Vaart and Wellner \citep{van1996weak}. Each of the classes $\{x:x_j-\theta_j \geq 0\}$ for $j=1,2,\dots,k$, is a sub-collection of VC classes of sets $\mathcal{C}_2=\{x:a^Tx+b \geq 0,\ a \in \mathbb{R}^k$,\ $b \in \mathbb{R}\}$. Hence by Lemma 2.6.15 of van der Vaart and Wellner \citep{van1996weak}, $\mathcal{C}$ forms a VC-class of sets. \par
Thus we proved that $\mathcal{F}_{1R}$ is a bounded VC major class of functions. Hence $\mathcal{F}_{1R}$ satisfies the uniform entropy condition. \par
For $\mathcal{F}_{2R}$, we see that the class of functions $\{x \mapsto \sign(\theta_j-x_j):\Vert\theta-\theta_0\Vert \leq R\}$ belongs to a finite dimensional vector space. Hence from Lemma 2.6.15 of van der Vaart and Wellner \citep{van1996weak}, $\mathcal{F}_{2R}$ is a bounded VC major class of functions and satisfies the uniform entropy condition. Therefore $\mathcal{F}_R=\mathcal{F}_{1R}\mathcal{F}_{2R}$ is $P_0$-Donsker.\par
Next we need to prove (7.5), that is, $\max_{1 \leq j \leq k} P_0(\psi_j(\cdot,\theta)-\psi_j(\cdot,\theta_0))^2 \rightarrow 0$ as $\theta \rightarrow \theta_0$. Note that $\psi_j(x,\theta) \rightarrow \psi_j(x,\theta_0)$ for every $x$ as $\theta \rightarrow \theta_0$ for $j \in \{1,\dots,k\}$. Also $(\psi_j(x,\theta)-\psi_j(x,\theta_0))^2 \leq 4$ for every $x$ and every $\theta$. Hence by the dominated convergence theorem, $P_0(\psi_j(\cdot,\theta)-\psi_j(\cdot,\theta_0))^2 \rightarrow 0$ as $\theta \rightarrow \theta_0$ for $j \in \{1,\dots,k\}$. Thus (7.5) is established.
\end{proof}
\begin{proof}[Verification of Condition \textup{3} in Theorem \textup{3.2}] 
For every $j \in \{1,2,\dots,k\}$ and $\theta \in \mathbb{R}^k$, $\psi_j(x,\theta)$ is bounded by $1$ and hence is square-integrable. The $(i,j)$th element of $\Sigma_0=P_0\psi(x,\theta_0)\psi^T(x,\theta_0)$ is given by
\begin{align}
\sigma_{ij}&=\int \frac{\vert x_i-\theta_{0i} \vert^{p-1}\vert x_j-\theta_{0j} \vert^{p-1}}{\Vert x -\theta_0 \Vert_p^{2(p-1)}}\sign(\theta_{0i}-x_i)\sign(\theta_{0j}-x_j)\diff P_0 \\
&\leq \int1 \diff P_0 <\infty. \nonumber
\end{align}
The class of functions $\{\psi_j(x,\theta):j=1,2,\dots,k,\Vert \theta-\theta_0 \Vert \leq R\}$ has a constant envelope $1$. Hence $D_n(x)$ defined in (3.8) is equal to $2$ and it satisfies (3.9).
\end{proof}
\begin{proof}[Verification of Condition $4$]
First we prove $\Vert \hat{\theta}_n -\theta_0 \Vert \overset{P_0}{\to} 0$. Note that $\hat{\theta}_n$ can be written as
\begin{equation}
\hat{\theta}_n= \argmax_{\theta}\mathbb{P}_nm_\theta,
\end{equation}
where $m_\theta(x)= -\Vert x-\theta \Vert_p+\Vert x \Vert_p$. Naturally the population analog of $\hat{\theta}_n$ is given by
\begin{equation}
\theta(P)= \argmax_{\theta}\mathbb{P}m_\theta.
\end{equation}
From Corollary 3.2.3 of van der Vaart and Wellner \citep{van1996weak}, we need to establish two conditions as follows:
\begin{enumerate}[label=(\alph*)]
\item $\sup_{\theta} \vert \mathbb{P}_nm_{\theta}-P_0m_{\theta} \vert \rightarrow 0$ in probability;
\item there exists a $\theta_0$ such that $P_0m_{\theta_0}>\sup_{\theta \notin G}P_0{m_\theta}$ for every open set $G$ containing $\theta_0$.
\end{enumerate}
The first condition can be proved by showing that the class of functions $\{m_{\theta}:\theta \in \mathbb{R}^k\}$ forms a $P_0$-Glivenko-Cantelli class.
From Theorem 19.4 of van der Vaart \citep{van2000asymptotic}, the class $\mathcal{M}=\{m_{\theta}:\theta \in \Theta \subset \mathbb{R}^k\}$ will be $P_0$-Glivenko-Cantelli if $N_{[\,]}(\epsilon,\mathcal{M},L_1(P_0)) < \infty$ for every $\epsilon>0$.\par
By Example 19.7 of van der Vaart \citep{van2000asymptotic}, for a class of measurable functions $\mathcal{F}=\{f_{\theta}:\theta \in \Theta \subset \mathbb{R}^k\}$, if there exists a measurable function $m$ such that
\begin{equation}
\vert f_1(x)-f_2(x) \vert \leq m(x)\Vert \theta_1-\theta_2 \Vert,
\end{equation}
for every $\theta_1,\theta_2$ and $P_0\vert m \vert ^r < \infty$, then there exists a constant $K$, depending on $\Theta$ and $k$ only, such that the bracketing numbers satisfy
\begin{equation}
N_{[\,]}(\epsilon \Vert m \Vert_{P_0,r},\mathcal{F},\mathcal{L}_r(P_0)) \leq K\left(\frac{\mathrm{diam}\ \Theta}{\epsilon}\right)^k,
\end{equation}
for every $0<\epsilon<\mathrm{diam}\ \Theta$.\par
To use Example 19.7 of van der Vaart \citep{van2000asymptotic}, we need to restrict the parameter space to a compact subset of $\mathbb{R}^k$. We show that this can be avoided in our case by asserting that the parameter space can be restricted to a sufficiently large compact set with high probability. We claim that if for some $0<\epsilon<1/4$ and $K>0$, $P_0$ on $(\mathbb{R}^k,\mathscr{R}^k)$ satisfying $P_0(\Vert X \Vert_p \leq K)>1-\epsilon$ for any probability measure $P_0$, then $\Vert \theta(P_0) \Vert_p \leq 3K$; here $\mathscr{R}^k$ denotes the Borel sigma field on $\mathbb{R}^k$. Define $M(P_0,\theta)=P_0m_\theta=P_0(\Vert X-\theta \Vert_p-\Vert X \Vert_p)$. We show that for $0<\epsilon<1/4$, there exists $K>0$ such that $\Vert \theta \Vert_p \geq 3K$ implies $M(P_0,\theta)>0$. If $\Vert X \Vert_p \leq K$ and $\Vert \theta \Vert_p \geq 3K$, then
\begin{equation*}
\Vert X-\theta \Vert_p \geq \Vert \theta \Vert_p - \Vert X \Vert_p \geq\frac{2\Vert \theta \Vert_p}{3}+K-\Vert X \Vert_p \geq \frac{2\Vert \theta \Vert_p}{3},
\end{equation*}
Hence as $\Vert X \Vert_p \leq K \leq \Vert \theta \Vert_p/3$,
\begin{equation*}
\Vert X-\theta \Vert_p-\Vert X \Vert_p \geq \frac{2\Vert \theta \Vert_p}{3}-\frac{\Vert \theta \Vert_p}{3}=\frac{\Vert \theta \Vert_p}{3}.
\end{equation*}
Now since always $\big\vert \Vert X -\theta \Vert_p - \Vert X \Vert_p \big\vert \leq \Vert \theta \Vert_p$, we can write
\begin{align*}
M(P_0,\theta) &= \int_{\Vert X \Vert_p \leq K}(\Vert X-\theta \Vert_p-\Vert X \Vert_p)\mathrm{d}P_0+\int_{\Vert X \Vert_p>K}(\Vert X-\theta \Vert_p-\Vert X \Vert_p)\mathrm{d}P_0\\
&\geq \Vert \theta \Vert_p(\frac{1}{3}P_0(\Vert X \Vert_p \leq K)-P_0(\Vert X \Vert_p>K )\big)\\
&=\Vert \theta \Vert_p(\frac{1}{3}-\frac{4}{3}P_0(\Vert X \Vert_p>K))\\
&\geq \Vert \theta \Vert_p(\frac{1}{3}-\frac{4}{3}\epsilon)>0.
\end{align*}
We assume that $0<\epsilon<1/4$ and $K>0$ have been chosen so that $P_0$ satisfies $P_0(\Vert X \Vert_p \leq K)>1-\epsilon$. Hence $\Vert \theta(P_0)\Vert_p \leq 3K$ with high probability. Also, since $\mathbb{P}_n \rightsquigarrow P_0$, for some $0<\epsilon<1/4$ and $K>0$, $\mathbb{P}_n$ satisfies $\mathbb{P}_n(\Vert X \Vert_p \leq K)>1-\epsilon$ with high probability. Hence $\Vert \hat{\theta}_n \Vert_p \leq 3K$ with high probability. Similarly, we know that $\mathbb{B}_n -\mathbb{P}_n \overset{P_0^{\infty}\times \mathbb{B}_n}\to 0$ in the weak topology. Hence with high joint probability, $\mathbb{B}_n$ satisfies $\mathbb{B}_n(\Vert X \Vert_p \leq K)>1-\epsilon$, leading to $\Vert \theta(\mathbb{B}_n) \Vert_p \leq 3K$ with high joint probability.
\par
Using Minkowski's inequality, we write
\begin{align*}
\vert m_\theta (x)-m_{\theta^{\prime}}(x)\vert =& \vert\Vert x-\theta^{\prime} \Vert_p-\Vert x-\theta \Vert_p \vert\\
\leq & \Vert \theta-\theta^{\prime}\Vert_p. 
\end{align*}
This expression is bounded by $\Vert \theta - \theta^{\prime} \Vert$ for $p \geq 2$, by the fact that $\Vert z \Vert_{p+a} \leq \Vert z \Vert_p$ for any vector $z$ and real numbers $a \geq 0$ and $p \geq 1$. Hence we choose $m(x)=1$ for every $x$ and therefore $P_0\vert m \vert =1$. This ensures that $N_{[\,]}(\epsilon,\mathcal{M},\mathcal{L}_1(P_0))<\infty$ and hence Condition (a) is satisfied. From Condition (C2) in Theorem 3.1, Condition (b) holds. Therefore $\hat{\theta}_n \rightarrow \theta_0$ in $P_0$-probability. \par
Now to prove the consistency of $\theta(\mathbb{B}_n)$, which is viewed as a \enquote{bootstrap estimator}, we use Corollary 3.2.3 in van der Vaart and Wellner \citep{van1996weak}. Two conditions are needed for proving this. The first condition is $\sup_{\theta}\vert \mathbb{B}_nm_{\theta}-P_0m_{\theta}\vert \overset{P_0 \times \mathbb{B}_n}{\to} 0$. We verify this condition using the multiplier Glivenko-Cantelli theorem which is given in Corollary 3.6.16 of van der Vaart and Wellner \citep{van1996weak}. By the representation $\mathbb{B}_n=\sum_{i=1}^n B_{ni}\delta_{X_i}$, where $(B_{n1},\dots,B_{nn}) \sim \mathrm{Dir}(n;1,\dots,1)$, it follows that $B_{ni} \geq 0$, $\sum_{i=1}^n B_{ni} =1$ and $B_{ni} \sim \mathrm{Be}(1,n-1)$. Therefore, for every $\epsilon>0$, as $n \rightarrow \infty$
\begin{equation*}
P\left(\max_{1\leq i \leq n} \vert B_{ni}\vert < \epsilon \right)
=\left(\int_0^\epsilon \frac{(1-y)^{n-2}}{B(1,n-1)}dy \right)^n
=(1-(1-\epsilon)^{n-1})^n
\rightarrow 1.
\end{equation*}
Thus the first condition is proved. The second condition is the same as the \enquote{well-separatedness} condition (b) which we already verified. 
 So, we have $\hat{\theta}_n \overset{P_0}\to \theta_0$ and $\theta(\mathbb{B}_n) \overset{P_0 \times \mathbb{B}_n}\to \theta_0$. Hence by an application of the triangle inequality, $\theta(\mathbb{B}_n) \overset{\mathbb{B}_n}\to \hat{\theta}_n$ in $P_0$-probability.
\end{proof}
\begin{proof}[Verification of Condition $5$]
It has already been mentioned that the Bayesian bootstrap weights satisfy the bootstrap weights $\mathrm{(\romannumeral 1)}$-$(\romannumeral 5)$. 
\end{proof}
\begin{proof}[Proof for arbitrary $p>1$ when $k=2$] 
This proof differs from the previous case only in two places. The main difference lies in the fact that $\mathcal{S}^1=\{S_{\theta,t}^1: \Vert \theta - \theta_0 \Vert \leq R,\ t \in \mathbb{R}\}$ is a VC class for any fixed $p>1$, where $S_{\theta,t}^1$ is as defined in (7.10). For $k=2$, $S_{\theta,t}^1$ can be written as
\begin{equation}
S_{\theta,t}^1=\Big \{x:\vert x_1-\theta_1 \vert >\left(\frac{t^{p/(p-1)}}{1-t^{p/(p-1)}}\right)^{1/p}\vert x_2-\theta_2 \vert\Big \}.
\end{equation}
Define $R_{\theta,c}^1 = \{x:\vert x_1-\theta_1 \vert >c\vert x_2-\theta_2 \vert\}$ and $\mathcal{R}^1 = \{R_{\theta,c}^1: \theta \in \mathbb{R}^2, c \in \mathbb{R}\}$.

It is enough to show that $\mathcal{R}^1$ is a VC class, since $\mathcal{R}^1$ contains $\mathcal{S}^1$. We can write $R_{\theta,c}^1$ as
\begin{align*}
R_{\theta,c}^1 = &\{x:(x_1-\theta_1)>c(x_2-\theta_2),\ x_1 \geq\theta_1,\ x_2 \geq\theta_2 \}\\
&\bigcup \{x:(\theta_1-x_1)>c(x_2-\theta_2),\ \theta_1 \geq x_1,\ x_2 \geq \theta_2\}\\
&\bigcup \{x:(x_1-\theta_1)>c(\theta_2-x_2),\ x_1\geq\theta_1,\ \theta_2\geq x_2\}\\
&\bigcup \{x:(\theta_1-x_1)>c(\theta_2-x_2),\ \theta_1 \geq x_1,\ \theta_2 \geq x_2\}.
\end{align*}
Define $C_{\theta,c}=\{x:(x_1-\theta_1)>c(x_2-\theta_2),\ x_1\geq\theta_1,\ x_2\geq\theta_2\}$. By the same argument used in the previous proof, it is enough to show that $\mathcal{C}=\{C_{\theta,c}:\theta \in \mathbb{R}^2,c \in \mathbb{R}\}$ forms a VC class of sets. This follows since  $C_{\theta,c}$ can be written as
\begin{equation*}
C_{\theta,c}=\{x:x_1-\theta_1>c(x_2-\theta_2)\}\bigcup \{x:x_1\geq \theta_1\}\bigcup \{x:x_2 \geq \theta_2\}.
\end{equation*}
Each of the sets in the right hand side of the above expression is a sub-collection of $\mathcal{C}_2=\{x:a^Tx+b \geq 0:a \in \mathbb{R}^2,b \in \mathbb{R}\}$ which is a VC class by Lemma 2.6.15 of van der Vaart and Wellner \citep{van1996weak}. Hence by Lemma 2.6.17 of van der Vaart and Wellner \citep{van1996weak}, $\mathcal{C}$ is also a VC-class.\par
The second difference arises in ensuring (7.17). Following the same proof as in the previous case, by an application of Minkowski's inequality, we write
\begin{equation*}
\vert m_{\theta}-m_{\theta^{\prime}} \vert
\leq \Vert \theta - \theta^{\prime} \Vert_p.
\end{equation*}
This expression is bounded by $2^{(1/p)-(1/2)}\Vert \theta - \theta^{\prime} \Vert$ for $1<p<2$ and $\Vert \theta - \theta^{\prime} \Vert$ for $p \geq 2$. 
\end{proof}
\begin{lemma}
The asymptotic posterior distribution of $\sqrt{n}(\theta(P)-\hat{\theta}_n)$ is the same as the asymptotic conditional distribution of $\sqrt{n}(\theta(\mathbb{B}_n)-\hat{\theta}_n)$.
\end{lemma}
\begin{proof}
We know $\theta(\mathbb{B}_n)$ satisfies $\Psi^{\star}(\theta(\mathbb{B}_n))=\mathbb{B}_n\psi(\cdot,\theta)=0$ and $\theta(P)$ satisfies $\Psi(\theta(P))=P\psi(\cdot,\theta)=0$.\par
The posterior distribution of $P$ given $X_1,\dots,X_n$ is $\mathrm{DP}(\alpha+n\mathbb{P}_n)$. From the fact that $\Vert P-\mathbb{B}_n \Vert_{TV}=o_{\mathrm{P_r}}(n^{-1/2})$ a.s. $[P_0^{\infty}]$, where $P_r= P^{\infty}\times \mathbb{B}_n$,
\begin{align*}
    \big\Vert P \psi(X,\theta)-\mathbb{B}_n\psi(X,\theta)\big\Vert &\leq \big\Vert \psi\big\Vert_{\infty}\big\Vert P-\mathbb{B}_n\big\Vert_{\mathrm{TV}}\\
    & \leq \big\Vert P-\mathbb{B}_n\big\Vert_{\mathrm{TV}}.
\end{align*}
where $$\Vert \psi \Vert_{\infty}=\sup_x\Vert \psi(x,\theta)\Vert \leq 1.$$
In view of this result, given $X_1,\dots,X_n$,
\begin{equation}
    \big\Vert\Psi^{\star}(\theta(P))-\Psi(\theta(P))\big\Vert=\big\Vert\Psi^{\star}(\theta(P))\Vert=o_{\mathrm{P_r}}(n^{-1/2}).
\end{equation}
Hence, for given $X_1,\dots,X_n$, $\theta(P)$ makes the bootstrap scores $\Psi^{\star}(\theta)$ approximately zero in probability. Therefore, given the observations $X_1,\dots,X_n$, $\theta(P)$ qualifies to be a sequence of bootstrap asymptotic $Z$-estimators. Theorem 3.1 in Wellner and Zhan \citep{wellner1996bootstrapping} (Theorem 3.2 in this paper) holds for any sequence of bootstrap asymptotic $Z$-estimators $\doublehat{\theta}_n$ that satisfies
\begin{equation}
    \big\Vert \Psi^{\star}(\doublehat{\theta}_n)\big\Vert = o_{\mathrm{Pr}}(n^{-1/2}).
\end{equation}
Since we verified all the conditions in Theorem 3.2 in our situation, in view of (7.20), the asymptotic posterior distribution of $\sqrt{n}(\theta(P)-\hat{\theta}_n)$ is same as the asymptotic conditional distribution of $\sqrt{n}(\theta(\mathbb{B}_n)-\hat{\theta}_n)$.
\end{proof}
\subsection{Proof of Theorem 5.1}
We use the same technique of approximating the posterior distribution of $P$ by a Bayesian bootstrap distribution. We view $\hat{Q}_n(u_1),\dots,\hat{Q}_n(u_m)$ as a $Z$-estimator satisfying the system of equations $\mathbb{P}_n\psi(\cdot,Q)=0$, where $\psi(\cdot,Q)=\{\psi_{lj}(\cdot,Q_{lj}): l=1,\dots,m,\ j=1,\dots,k\}$ is given by
\begin{equation}
\psi_{lj}(x,Q_{lj})= \frac{\vert x_j-Q_{lj} \vert^{p-1}}{\Vert x-Q_l\Vert_p^{p-1}}\sign(Q_{lj}-x_j)+u_{lj}.
\end{equation}
We define $Q_{\mathbb{B}_n}(u_1),\dots,Q_{\mathbb{B}_n}(u_m)$ as the corresponding \enquote{Bayesian bootstrapped} versions of the $Z$-estimators $\hat{Q}_n(u_1),\dots,\hat{Q}_n(u_m)$. We use Theorem 3.1 of Wellner and Zhan \citep{wellner1996bootstrapping} (Theorem 3.2 in this paper) and conclude Theorem 5.1. The proof is exactly same as that of Theorem 3.1. and hence is omitted.\par
\bibliographystyle{acm}
\bibliography{sample}
\end{document}